\documentclass[12pt]{article}

\newcommand{\RED}[1]{{\color{red}{#1}}}

\usepackage{amssymb}
\usepackage{amsfonts}
\usepackage{latexsym}
\usepackage{graphicx}

\usepackage[usenames]{color}
\usepackage{amsmath,amsthm}

\usepackage{epsfig}

\usepackage{color}
\usepackage{amsmath,amsbsy,amsfonts,amssymb,latexsym,graphics,epsfig}

\newcommand{\ignore}[1]{}

\newcommand{\beqn}{\begin{eqnarray*}}
\newcommand{\eeqn}{\end{eqnarray*}}

\newcommand{\eps}{\varepsilon}

\newcommand{\ee}{\mathrm{e}}

\newcommand{\Z}{\mathbb{Z}}

\newcommand{\R}{\mathbb{R}}

\newcommand{\shf}{\mbox{\footnotesize $\frac{1}{2}$}}

\newcommand{\sot}{\mbox{\footnotesize $\frac{1}{3}$}}

\newcommand{\sq}{\mbox{\footnotesize $\frac{1}{4}$}}

\newcommand{\dd}{\mathrm{d}}

\newcommand{\Matrixc}[1]{\ensuremath{\left(\begin{array}{ccccccccccccccccr} #1 \end{array}\right)}}
\newcommand{\Matrix}[1]{\ensuremath{\left[\begin{array}{ccccccccccccccccr} #1 \end{array}\right]}}

\newcommand{\etal}{{\em et al.}}

\newcommand{\CC}{{\mathcal C}}

\newcommand{\GG}{{\mathcal G}}
\newcommand{\HH}{{\mathcal H}}

\newcommand{\MM}{{\mathcal M}}
\newcommand{\NN}{{\mathcal N}}

\newcommand{\TT}{{\mathcal T}}

\newcommand{\tr}{\mbox{tr}}

\newtheorem{theorem}{Theorem}[section]
\newtheorem{lemma}[theorem]{Lemma}
\newtheorem{definition}[theorem]{Definition}

\newtheorem{proposition}[theorem]{Proposition}
\newtheorem{remark}[theorem]{Remark}

\newtheorem{example}[theorem]{Example}

\renewcommand{\epsilon}{\varepsilon}
\newcommand{\trace}{\mbox{{\rm tr}}}

\newcommand{\ii}{\mbox{i}}

{\begingroup

  \begin{enumerate}}
  {\end{enumerate}\endgroup}

\newcommand{\END}{\hfill\mbox{\raggedright$\Diamond$}}

\title{Stable Synchronous Propagation in\\
Feedforward Networks for Biped Locomotion}
\author{Ian Stewart  and David Wood\\ Mathematics Institute\\ University of Warwick \\ Coventry CV4 7AL
\\ United Kingdom}

\date{\today}

\begin{document}
\maketitle

\begin{abstract}
Rhythmic gait patterns in animal locomotion are widely believed to be
produced by a central pattern generator (CPG), a network of neurons that drives
the muscle groups. In previous papers we have discussed how phase-synchronous
signals can propagate along chains of neurons using a feedforward lift
of the CPG, given sufficient conditions for stability to synchrony-breaking
perturbations, and shown that stable signals are common for
 four standard neuron models. Here we
apply these ideas to biped locomotion using a fifth model: Wilson--Cowan (or rate model)
neurons. Feedforward architecture propagates the 
phase pattern of a CPG along a chain of identical modules.
For certain rate models, we give analytic conditions that are sufficient for transverse
Liapunov and Floquet stability.  We compare different notions 
of transverse stability, summarise some
numerical simulations, and outline an application to signal propagation in
a model of biped locomotion.
\end{abstract}

%\newpage
%\tableofcontents
%\newpage

\section{Introduction}
\label{S:intro}

The observation that legged animals move using a variety of rhythmic patterns
goes back at least to Aristotle \cite{A36}, who wondered whether a trotting horse can be
completely off the ground during some stages of its motion. His answer (`no') was
disproved by Muybridge \cite{M99}, who invented a camera with a rapid shutter
and used lines of such cameras, triggered by tripwires, to photograph
animals and humans in motion. It is often claimed that this allowed Leland Stanford Jr. to
win a large wager (a figure of \$25,000 to \$50,000 is typical), but 
Stanford was not a betting man and no money was involved \cite{D17}.
However, the same source cites evidence that he was in a dispute with 
Frederick McCrellish about this question.

Such patterns of movement are called {\em gaits} \cite{G68, G74}. 
They also occur for organisms such as arthropods 
(especially centipedes and millipedes), snakes, worms, and lampreys.
It is widely accepted that the underlying rhythms of gaits arise from the dynamics of
a {\em central pattern generator} (CPG), a neural `circuit' that naturally produces
oscillatory patterns.  A recent survey is \cite{B19}, and the prevailing paradigm
of alternating activity in excitatory and inhibitory neurons
is laid out in \cite{AC18,MR08}. Observations challenging some aspects of 
this paradigm are described in \cite{LPVB22}; they do not affect this paper.

CPG networks have been observed in some organisms,
and inferred in many others. In vertebrates the CPG appears to reside in
the spinal cord. The CPG rhythms drive the muscle groups that cause motion.
Many gaits are `symmetric' \cite{M68}: the movements of distinct legs are synchronised, or
are related by phase shifts that are a simple fraction of the overall period.
An example, the quadruped walk gait, is shown in Figure~\ref{F:ele_walk_outline}.
The legs move in the sequence left rear, left front, right rear, right front,
separated by one quarter of the period of the overall gait cycle.

\begin{figure}[h!]
\centerline{%
\includegraphics[width=0.75\textwidth]{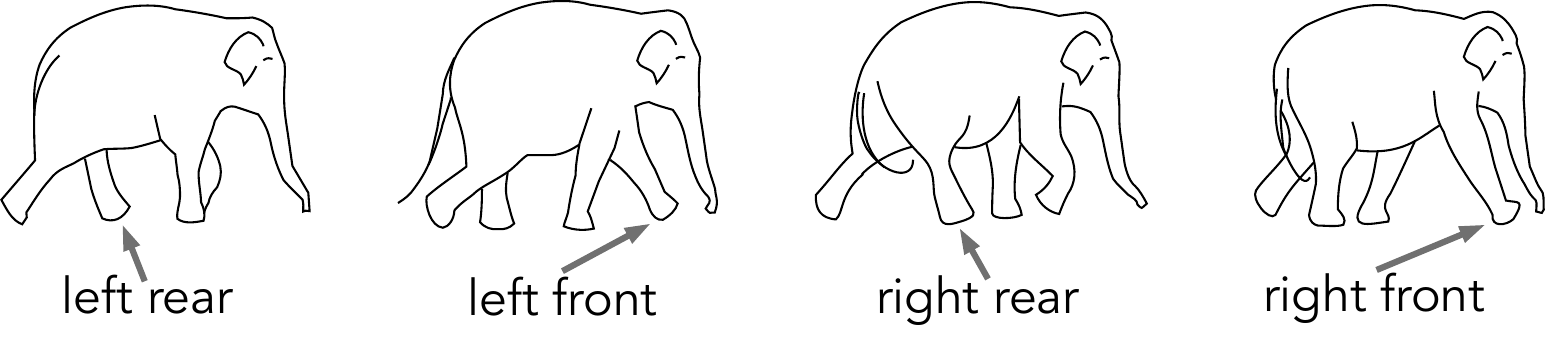}
}
\caption{Walk of an elephant: successive legs hit the ground at $\sq$-period intervals.}
\label{F:ele_walk_outline}
\end{figure}

\subsection{Biped Gaits}
In this paper we examine a simpler case, that of bipeds. 
The most obvious CPG model for $n$-legged animals is a network with $n$
nodes, whose topology is constructed to create the phase 
patterns typically observed in gaits \cite{CS93a,CS93b}. 
However, such a model for quadrupeds ($n=4$) predicts that the trot and
pace gaits are `dynamically conjugate'---related by an element of the symmetry
group of the model ODE---so they should coexist for all parameter
values. This is inconsistent with observations---for example, some breeds of horse,
especially those bred for harness racing,
can trot but not pace \cite{H93}. A proposed modification \cite{GSBC98,GSCB99}, derived from known features of quadruped gaits, 
doubles the number of nodes to $2n$. Each leg corresponds to two 
nodes of the network, interpreted as controlling two muscle
groups: flexors and extensors. The resulting CPG architecture
has $\Z_n \times \Z_2$ symmetry: an $n$-cycle on each side of the animal
combined with left-right symmetry. This $n$-cycle controls the $n/2$ legs
on the corresponding side of the animal, so there are two nodes per leg.
This topology
leads to several predictions that are
consistent with observations, including half-integer wave numbers in
arthopod gaits.

\begin{figure}[h!]
\centerline{%
\includegraphics[width=.4\textwidth]{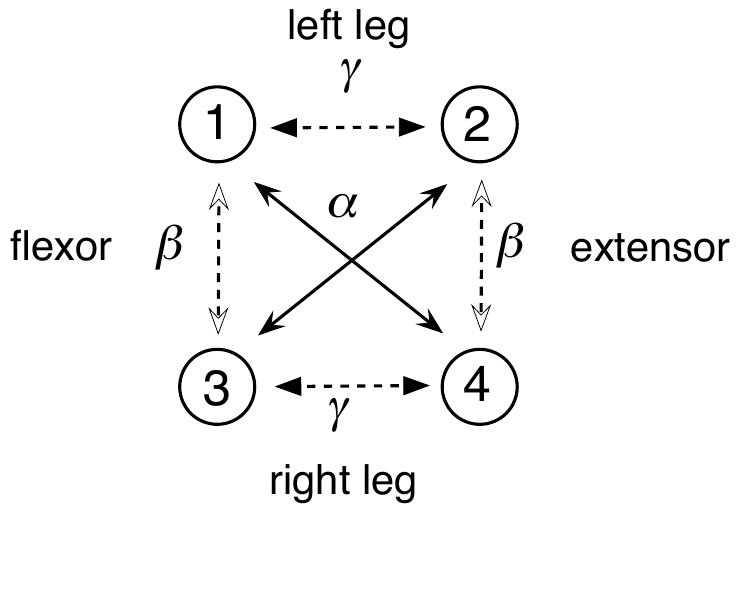} 
}
\caption{A 4-node CPG network for bipedal gaits, proposed in \cite{PG06}.}
\label{F:4CPG}
\end{figure}

Pinto and Golubitsky \cite{PG06} analysed the corresponding
network model of a CPG for bipeds ($n=2$): 
a 4-node CPG with $\Z_2 \times \Z_2$ symmetry,
Figure \ref{F:4CPG}. We have rotated their figure through a 
right angle and relabelled the nodes, for consistency with later networks.
In their model, two nodes control extensor and flexor muscle groups 
of the ankle of the left leg, and
the other two control extensor and flexor muscle groups of the ankle of the right leg.
The activity of these muscles and the corresponding ankle movements are idealised,
so that the muscles either act in synchrony or are separated by a 
half-period phase shift. This idealisation is reasonably close to
reality, see \cite{M82, M86}.
Other muscle groups of course play a role, but are not
represented in the model. 

The $H/K$ Theorem of \cite{BG01} (originally motivated by gait patterns
and generalised to arbitrary symmetric dynamical systems)
classifies the possible phase patterns in a dynamical system with symmetry
group $\Gamma$ in terms of pairs of subgroups $H \subseteq K$ of $\Gamma$.
Here $K$ is the group of pointwise symmetries of the periodic orbit and
$K$ is the group of setwise symmetries. The main condition is that
$K$ must be a normal subgroup of $H$ with a cyclic group quotient.
Other technical conditions are also required.

For the network of Figure \ref{F:4CPG} the symmetry group is $\Gamma =\Z_2 \times \Z_2$, described in Section \ref{S:SCPG}.
The $H/K$ Theorem leads to a list of 11 pairs $(H,K)$, corresponding to
a total of $11$ gaits. Of these, four are `primary' gaits,
which can occur by generic Hopf bifurcation from the trivial steady state
with all legs stationary. These gaits are the (two-legged) {\em hop}, {\em walk}, (two-legged) {\em jump},
and {\em run}, see Table \ref{T:primary_gaits}. Again, Section \ref{S:SCPG}
provides further details.
In primary gaits, all nodes
have the same waveform, but with regular phase differences. (This remark
applies to the model; in reality the waveforms are only approximately the same.)
The remaining
seven `secondary' gaits involve either two or four different waveforms. Six
of them can be viewed as `mode interactions' between two primary gaits.

\begin{table}[!htb]
\begin{center}
\begin{tabular}{|c|c|c|c|c|}
\hline
gait & (1) left flexor & (2) left extensor & (3) right flexor & (4) right extensor\\
\hline
\hline
hop & $x(t)$ & $x(t)$ & $x(t)$ & $x(t)$ \\
walk & $x(t)$ & $x(t+\shf)$ & $x(t+\shf)$ & $x(t)$ \\
jump & $x(t)$ & $x(t+\shf)$ & $x(t)$ & $x(t+\shf)$ \\
run & $x(t)$ & $x(t)$ & $x(t+\shf)$ & $x(t+\shf)$ \\
\hline
\end{tabular}
\caption{Classification of primary gaits for bipedal locomotion, assuming
a CPG as in Figure \ref{F:4CPG}.}
\label{T:primary_gaits}
\end{center}
\end{table}

A rate model for this CPG is analysed in \cite{S14}, and the results
are summarised in Section \ref{S:RE4CPG}.

The main objective of this paper is to use the methods of
\cite{SW23a, SW23b} to investigate alternative
networks that generate the same patterns as the 4-node CPG in Figure \ref{T:primary_gaits}, 
but propagate them along arbitrarily long chains. One of these
networks is a `feedforward lift' of the 4-node CPG, Figure \ref{F:4CPGcascade} (top).
The other is {\em almost} a feedforward lift, except for lateral
connections, as in Figure \ref{F:4CPGcascade} (bottom). The top network is
slightly more tractable mathematically. The bottom one is biologically
more plausible and we analyse it using similar methods.

The key feature that we analyse is the {\em stability} of these patterns;
not just to synchrony-preserving perturbations, but also to
synchrony-breaking perturbations. We consider both
Floquet and Liapunov stability, as well as transverse stability of the
synchrony subspace; see Section \ref{S:STS} for a summary of these notions.
 
\begin{figure}[htb]
\centerline{%
\includegraphics[width=.8\textwidth]{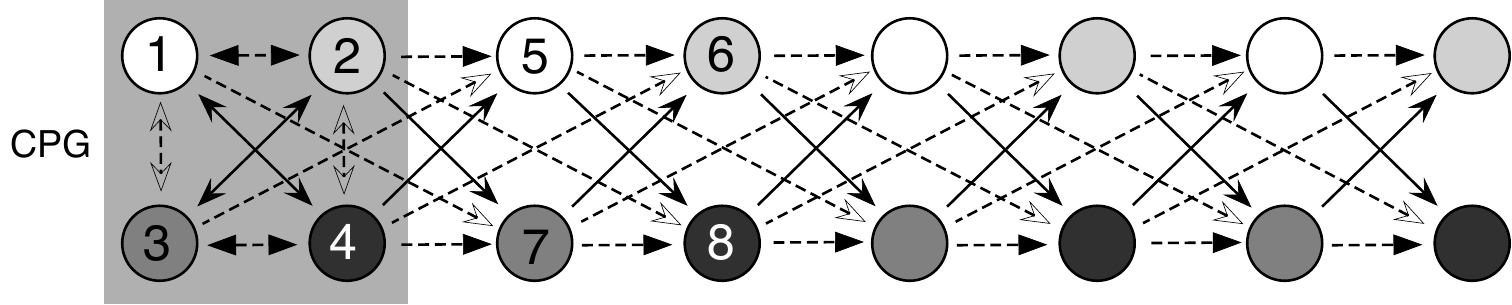} }
\vspace{.3in}
\centerline{%
\includegraphics[width=.8\textwidth]{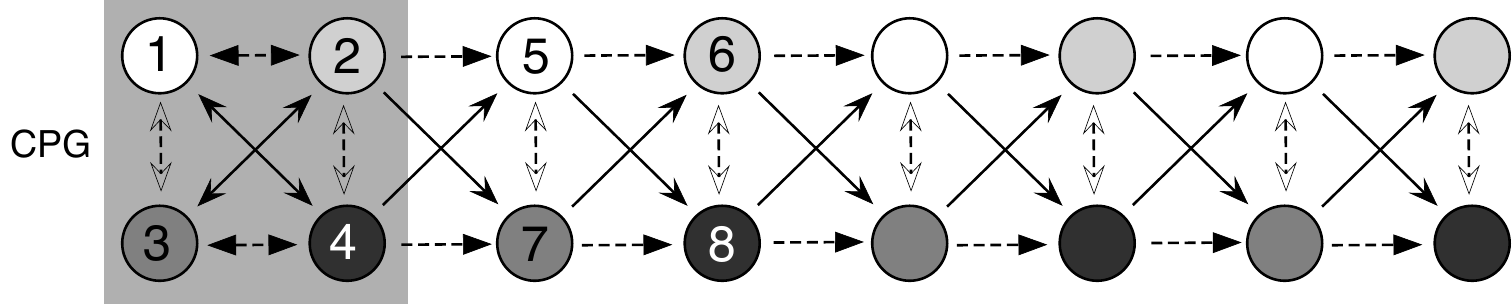} }
\caption{Two lifts of
Figure~\ref{F:4CPG}\protect, which therefore support the same rhythms. In each
subfigure, the top row of nodes corresponds to the left side of the animal
and the bottom row to the right side.
{\em Top}: Feedforward lift. {\em Bottom}: Lift in which all connections to the chain are
feedforward except for the lateral ones (vertical double-headed arrows).}
\label{F:4CPGcascade}
\end{figure}

\subsection{Synchrony and Phase Patterns}

Regular phase patterns are common in networks of dynamical systems
with suitable symmetries \cite{GS02,GS23,GSS88}. This suggests modelling
gait patterns using network dynamics \cite{CS93a,CS93b}.
In network dynamics, each network (labelled directed graph) has an associated set of 
{\em admissible} ODEs --- ordinary differential equations that are
compatible with the network structure in a precise technical sense.
In this paper we state explicitly the admissible ODEs for the networks
under discussion, so we do not repeat the general definition of
`admissible'. Full details are in \cite{GS23,GST05,SGP03}. There is a short summary in \cite{SW23a}
and a briefer one in \cite{SW23b}.

For modelling purposes, we follow a common convention in the
literature and adopt a strong definition of synchrony.
Two nodes are {\em synchronous}, for some admissible ODE, if their waveforms
(time series) are identical; they are {\em phase-synchronous} if their waveforms
are identical except for a phase shift (time translation).
These definitions are idealisations, but they open up a
powerful mathematical approach with useful implications. Real
systems can be considered as perturbations of
idealised ones, and much of the interesting structure persists in an appropriately
approximate form \cite[Section 8]{SW23b}.

A {\em synchrony pattern} partitions the nodes of the network
into sets that are synchronous, often called {\em clusters}. 
A {\em phase pattern} partitions the nodes of the network
into sets that are phase-synchronous, and also assigns specific phase shifts as
fractions of the overall period. These patterns are {\em rigid} if they persist after all
sufficiently small admissible perturbations of the ODE. In the phase-synchronous case
the period changes under perturbation, so we require
the phase shifts to be preserved when measured as fractions of the period. In the context of
symmetric dynamics, these `spatiotemporal' symmetries are encoded in
two subgroups $H$ and $K$ of the symmetry group $\Gamma$, where
$K$ is a normal subgroup of $H$ and $H/K$ is a finite cyclic group \cite{B01,GS02}.
The group $K$ determines which nodes are synchronous with each other; 
the group $H$ determines the
phase shifts when considered modulo $K$. In a network version,
$K$ becomes a so-called balanced colouring and $H$ is a group of symmetries
of the corresponding quotient network \cite{GS23}. 

The groups $H$ and $K$ for Table \ref{T:primary_gaits} are listed in \cite{PG06}
and are easily inferred from the table shown. In all cases $H = \Z_2 \oplus \Z_2$
and $K$ is equal to $H$ for the hop and otherwise is one of the three subgroups
that are cyclic of order 2. See Example \ref{ex:biped}.

\begin{example}\em
\label{ex:elephant_walk}
The walk gait of Figure \ref{F:ele_walk_outline} affords a simple example of a phase pattern. Denote the legs by
$LR, LF, RR, RF$ (left rear, left front, right rear, right front). Legs move in that sequence,
with successive $\sq$-period phase shifts. The synchrony pattern $K$ is trivial:
distinct legs do not synchronise. The symmetry group $H$ is cyclic of order 4,
permuting the legs in the cycle $(LR, LF, RR, RF)$. A generator of $H$ 
(move one place along the cycle) corresponds to a phase shift $T/4$ where $T$
is the overall period.

The bound, pace, and trot gaits (observed in other quadrupeds, \cite{B01,BG01,CS93a,G74,GSBC98,GSCB99}) also
have nontrivial synchrony/phase patterns. In the bound, the back legs move together and then
the front legs move together, so the synchrony clusters $K$
are $\{LR,RR\}$ and $\{LF,RF\}$. The symmetry group of the clusters is
cyclic of order 2, and the phase difference between the clusters 
is half the period.

The pace and trot are similar, but the synchrony clusters $K$
are $\{LR,LF\}$ and $\{RR,RF\}$ for the pace and $\{LR,RF\}$ 
and $\{RR,LF\}$ for the trot. Again the symmetry group of the clusters is
cyclic of order 2, and the phase difference between the clusters 
is half the period.
\end{example}

\subsection{Symmetries and CPGs}
\label{S:SCPG}

Phase patterns arise naturally in symmetric dynamical systems \cite{GSS88},
leading to the suggestion that gait CPGs should also have appropriate symmetries \cite{CS93b}.
%Gait analysts intuited this relationship long ago when referring to `symmetric gaits'.

\begin{example}\em
\label{ex:biped}
Pinto and Golubitsky \cite{PG06,S14} propose
the 4-node network of Figure~\ref{F:4CPG} as a schematic CPG
for biped locomotion. A similar model,
potentially with less symmetry but the same topology, appears in \cite{AKFKST12}.
The nodes represent neurons or, more realistically,
populations of neurons \cite{WC72}. Although a biped has two legs, the CPG has
four nodes. The muscle groups in each
leg are activated by two nodes: one flexor and one extensor.
This network has a symmetry group of order 4,
comprising the permutations
\beqn
&&{\rm id}\ = \Matrixc{1 & 2 & 3& 4 \\ 1 & 2 & 3 & 4}\qquad 
\rho = \Matrixc{1 & 2 & 3& 4 \\ 2 & 1 & 4 & 3}\qquad  \\
&&\sigma =\Matrixc{1 & 2 & 3& 4 \\ 3 & 4 & 1 & 2}\qquad 
\rho\sigma = \Matrixc{1 & 2 & 3& 4 \\ 4 & 3 & 2 & 1}
\eeqn
That is: the identity, swap left/right, swap flexor/extensor, swap diagonally.
These four permutations constitute the group $\Z_2 \times \Z_2$,
which corresponds naturally to the observed phase patterns.
\end{example}

CPG architectures of this kind have been explored for bipeds \cite{PG06,S14},
quadrupeds \cite{B01,CS93b,GSBC98,GSCB99,PG06,SJK90}, hexapods \cite{CS93a}, and arthropods \cite{GSCB99}, among others.
The main techniques are taken from symmetric dynamics: the symmetric Hopf bifurcation
theorem \cite{GSS88} and the $H/K$ theorem \cite{B01}. The CPGs that occur in these
papers generally have small numbers of neurons (or populations thereof). Typically
there are two neurons per leg, for mathematical reasons explained in   \cite{GSBC98,GSCB99}. A
possible biological interpretation is that one neuron (population) is required for 
extension and one for flexure of the corresponding muscle group.

\subsection{Classification of Phase Patterns}

The assumption of a symmetric CPG is natural in the context of dynamics 
with symmetry.
It is well known that symmetric dynamical systems can generate
rigid phase patterns, meaning that the phase shifts concerned remain
unchanged (as fractions of the period) after a small 
symmetry-preserving perturbation \cite{GSS88}. Rigidity is important in
biological applications, where precise model equations are seldom available,
so meaningful predictions should not depend on
fine details of the model. However,
 there are significant differences between a general symmetric dynamical system
and an admissible ODE for a symmetric network \cite[Chapter 16]{GS23}.
In particular, the network case permits a more general mechanism to generate
rigid phase patterns, without assuming that the network itself is
symmetric. Instead, the symmetry occurs in
a {\em quotient network} in which synchronous or phase synchronous nodes are identified.
We give a simple example in Section \ref{S:5node}; this network has only
trivial symmetry but it has Figure \ref{F:4CPG} as a quotient, which has
$\Z_2 \times \Z_2$ symmetry.
Thus these network architectures need not correspond directly to
physiological `wiring diagrams' (connectomes) in the animal, but could in principle be deduced
from wiring diagrams \cite{MBSS25}.

The rigid phase patterns that can occur in an admissible ODE for a network $\GG$
are characterised by the Network $H/K$ Theorem \cite[Theorem 17.19]{GS23}.
This theorem is a consequence of several highly plausible 
Rigidity Conjectures \cite{GS23}. These conjectures have
been proved for a wide variety of networks,
including all homogeneous and fully inhomogeneous networks \cite{GRW10,GRW12}.
However, those papers makes a tacit assumption that is not always valid \cite{GS23},
so the proof does not apply in full generality. The Rigidity Conjectures
have been proved for all networks under additional mild technical conditions \cite{S20overdet}.

Here $K$ is a balanced colouring of $\GG$, determining a rigid synchrony pattern,
and $H$ is a cyclic group $\Z_r$ of symmetries of the quotient network $\GG^K$ in which
synchronous nodes are identified. 
This kind of symmetry induces phase patterns
in which a generator of $\Z_r$ corresponds to a phase shift of $mT/r$ where $T$
is the period and $m$ is an integer between $0$ and $r-1$.
The network $H/K$ Theorem can be viewed as a catalogue of the possible 
phase patterns; any specific admissible ODE
`selects' patterns from this catalogue.

\subsection{Feedforward Lifts}

If a network $\mathcal{Q}$ is a quotient of a network $\GG$, we say that $\GG$
is a {\em lift} of $\mathcal{Q}$. Reversing the relation in this manner lets us start from
a network $\mathcal{Q}$ with certain dynamics, such as phase patterns,
and construct various lifts $\GG$
in which the corresponding clusters have the same dynamic behaviour
as the nodes of $\mathcal{Q}$. 
The Network $H/K$ Theorem describes a general mathematical scenario in which
a phase pattern generated by a simple symmetric network can be lifted to
a more complex asymmetric network.
The main aim of this paper is to investigate this possibility in the context of the
biped CPG of Figure \ref{F:4CPG}, seeking lifts that are biologically more plausible.
We focus on a crucial issue: the stability of the lifted state.
The basic construction used here is a {\em feedforward lift}, see Section \ref{S:FFL}.
This construction is discussed in \cite{SW23a} from a general 
mathematical viewpoint, and \cite{SW23b} 
analyses a series of `toy' examples based on standard model neurons,
again focusing on stability. 

The structure of a feedforward lift is motivated by animal physiology.
In many animals, oscillations propagate along
a chain of neurons, which oscillate in synchrony or phase-synchrony.
This architecture generates propagating time-periodic patterns.
Synchrony corresponds to a standing wave; phase synchrony to a travelling wave.
Legged locomotion in animals is the key example for this paper.
 Other examples are locomotion in {\em Drosophila} larvae \cite{GBEE13}
 and the nematode {\em Caenorhabditis elegans} \cite{OIB21,SSSIT21},
 peristaltic waves in
the intestine \cite{CBT08,CTB13,F06,Gr03,KF99}, and the heartbeat of the leech
\cite{CNO95,CP83,TS76,WHC04}.

Such patterns are found elsewhere: for instance
to propel continuum robots for medicine 
\cite{JC19,SOWRK10,ZHX20}
or exploration \cite{MWXLW17}. Gene regulatory networks (GRNs) also exhibit
synchrony patterns, and
chains of regulatory links can act as delay lines (Leifer \etal\ \cite{LMRASM20}) leading to phase patterns
The states of GRNs are usually equilibria, but periodic oscillations 
are also found \cite{PSGB10}.
Elowitz and Leibler \cite{EL00} found phase-synchronous patterns in the 
repressilator, a synthetic genetic circuit. In an idealised model these patterns
are a consequence of $\Z_3$ symmetry \cite{MBSS25,SRM24}.

\subsection{Example: the Leech Heartbeat}

Figure \ref{F:leech} shows a neural circuit for the heartbeat 
of the medicinal leech {\em Hirudo medicinalis} based on observations 
in Kristan \etal\ \cite{KCF05}. A CPG  (shaded) with 14 interneurons (HN) generates 
rhythmic signals that propagate along two chains of heart excitatory neurons (HE).
This network is a feedforward lift of the CPG, except for the HE nodes 3--7,
which have slightly different inputs from HE nodes 8 onwards.

A curious feature of the leech's heartbeat is that contractions occur as a standing wave
along one side of the animal, but as a travelling wave along the other. After
about 20--40 heartbeat cycles the two patterns swap sides \cite{CP83,KCF05}.
We do not investigate this kind of switching here; see \cite{BP04,PS08b} for
two possible explanations.

\begin{figure}[h!]
\centerline{%
\includegraphics[width=0.3\textwidth]{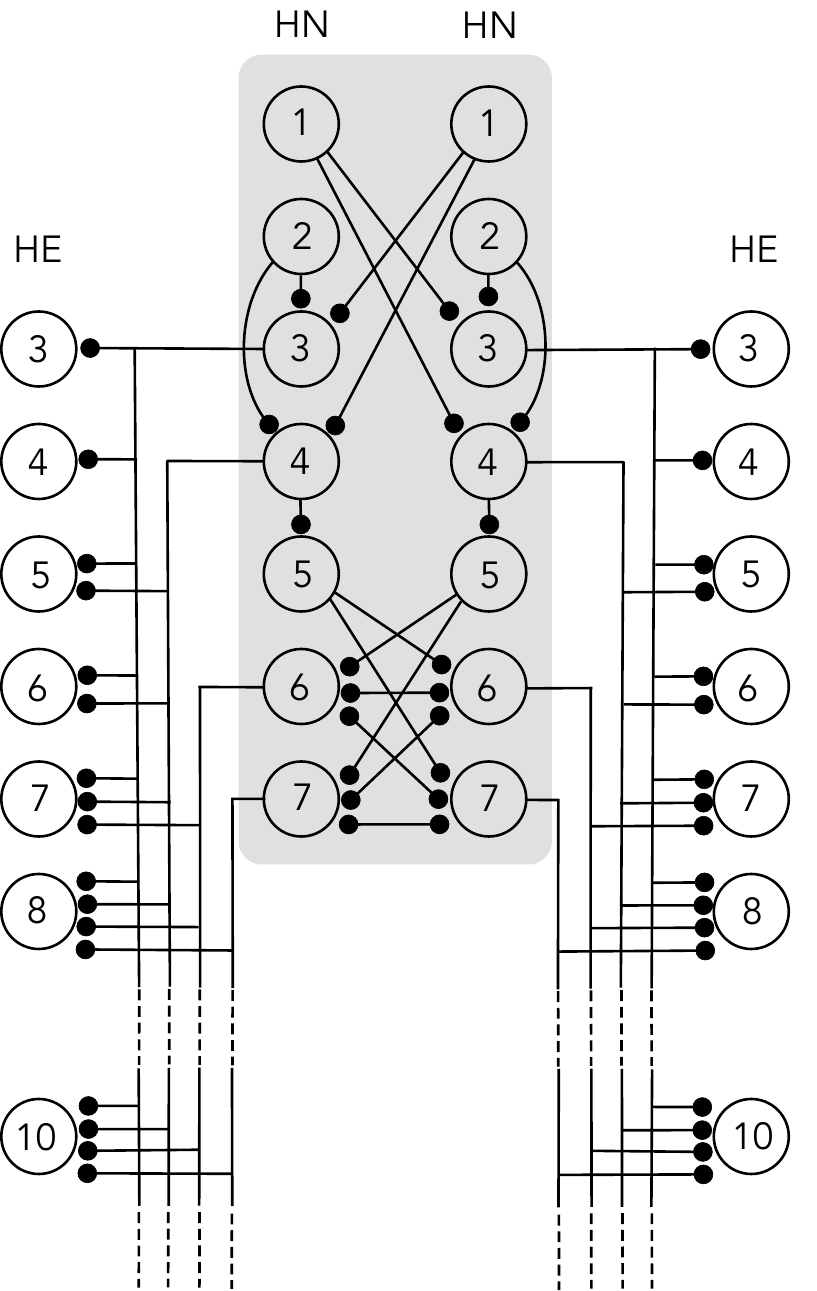}
}
\caption{Neural circuit for the heartbeat of the medicinal leech, after \cite{KCF05}.
Nodes in the shaded box form the CPG. HN = heart interneuron, HE = heart excitatory neuron.}
\label{F:leech}
\end{figure}

\subsection{Asymmetric CPGs can Generate Phase Patterns}
\label{S:5node}

The feedforward lifts in Figure \ref{F:4CPGcascade} preserve the $\Z_2 \times \Z_2$ symmetry
of the CPG, but asymmetric networks can also generate rigid phase
patterns via the same `quotient network' mechanism. Figure \ref{F:5node}
shows a 5-node network with trivial symmetry; it has one type of node
and three types of arrow. The map
\beqn
&&\varphi: \{1,2,3,4,5\} \to \{1,2,3,4\} \\
&& \varphi(1) = \varphi(5) = 1 \qquad \varphi(2) = 2 \qquad \varphi(3) = 3 \qquad \varphi(4) = 4
\eeqn
is a quotient map (\cite[Section 10.7]{GS23}), also called a graph fibration
\cite{BV02}, since it preserves node types, arrow types, and input sets
up to isomorphism. It therefore supports all types of phase pattern
associated with the $\Z_2 \times \Z_2$ symmetry of Figure \ref{F:4CPG},
subject to the condition that nodes
1 and 5 are synchronous.
However, the symmetry group
of the  $5$-node network of Figure \ref{F:5node} is trivial.
Specifically, there is precisely one unidirectional arrow 
of each type, namely $2 \to 5, 3 \to 5$, and $4 \to 1$, so the heads and tails
of these arrows are fixed by any symmetry. But these account for all five nodes.

In principle, the extra node can cause transverse instability, so
the parameter values for which these lifted gaits are stable
might differ from those for the original ones. Simulations show that
they can be stable for the 5-node network, for suitable parameters.
We have not investigated this issue further.

\begin{figure}[htb]
\centerline{%
\includegraphics[width=.3\textwidth]{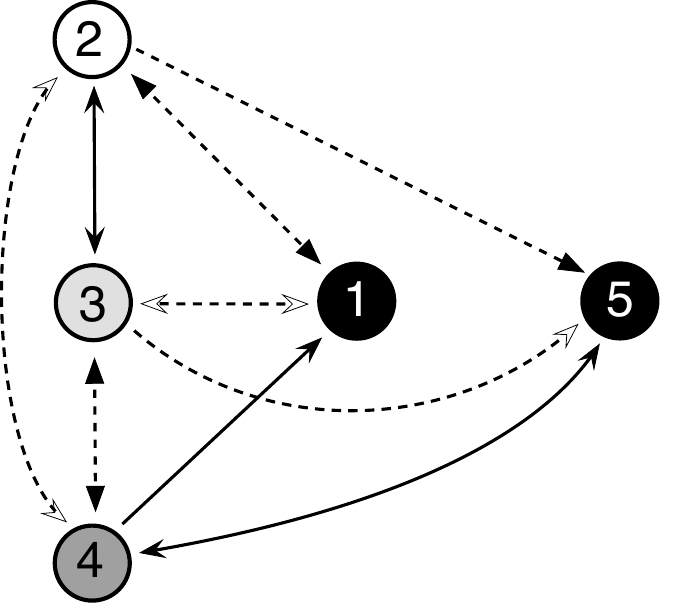} 
}
\caption{A $5$-node lift of Figure \ref{F:4CPG} with trivial symmetry. The synchrony pattern
illustrated, with clusters $\{1,5\}, \{2\}, \{3\},\{4\}$, determines a quotient network
with $4$ nodes, isomorphic to Figure \ref{F:4CPG}. Therefore the $5$-node
network supports the same periodic states (though not necessarily with the same stabilities) as the CPG in Figure \ref{F:4CPG}, but with nodes
1 and 5 are synchronous.}
\label{F:5node}
\end{figure}

\subsection{Comparison of Different Neuron Models}

Many different types of neuron model have been studied: Boolean switches, 
Pitts--McCulloch networks, PDEs,
integro-differential equations, delay differential equations, stochastic models\ldots 
\ Surveys and further details can be found in many sources, for example \cite{GK02,GKNP14,I10,LOHR99,NKMS25,RE98,wiki25}.
We comment briefly on five standard types of model neuron that are defined by ODEs. 

The first ODE model of a neuron is arguably that of Lapicque \cite{L07} in 1907.
Like many later ones, it was based on an electronic analogy.
A biologically more realistic ODE model appeared in the Nobel prizewinning
work of Hodgkin and Huxley \cite{HH52}. The objective was to model the initiation
and propagation of action potentials in the giant axon of the
longfin inshore squid {\em Doryteuthis pealeii}, and to obtain
a quantitative fit to data obtained from voltage clamp experiments.
Their model incorporates ionic mechanisms and is based on cable
theory. A major feature was an accurate prediction of the
speed of propagation of the nerve impulse. The Hodgkin--Huxley equations
are mathematically complex and cumbersome, 
involving four dynamic variables and numerous parameters, and
are best suited for numerical simulation.

To gain insight into key aspects of the dynamics generating a neuron impulse,
FitzHugh \cite{F61} and Nagumo \etal\ \cite{NAY62} introduced a mathematically
simpler model that preserves some key features of the Hodgkin--Huxley
model. It is a relaxation oscillator, generalising the classical
Van der Pol oscillator.
This model has only two variables, making it amenable to
standard methods of analysis such as phase-plane diagrams. It provides
geometric mechanisms for the activation and deactivation of a single pulse, but it
sacrifices quantitative accuracy.

Morris and Lecar \cite{ML81} proposed an alternative simplification
of the Hodgkin--Huxley model to represent voltage clamp experiments on
the barnacle {\em Balanus nubilus}. Like the FitzHugh--Nagumo model, 
their equation has two dynamic variables.

Hindmarsh and Rose \cite{HR84} devised yet another simplification of
the Hodgkin--Huxley model, designed to represent both single pulses and
bursting. It has three dynamic variables: the membrane potential and two 
representing transport through ion channels. It allows a more diverse
range of dynamic behaviour, including chaos, and provides qualitative versions
of the main types of dynamic behaviour found in experiments.

Wilson and Cowan \cite{WC72} formulated equations for populations of
excitatory and inhibitory neurons, whose mathematical formulation is
well suited to modelling networks of such neurons. Their model proposes
a pair of integro-differential equations for the time-evolution of
the numbers of excitatory and inhibitory neurons. A coarse-grained
version leads to a type of ODE known as a
{\em rate model} \cite{ET10}, discussed in Section \ref{S:RM}.
Here the state of each neuron in the network is represented by two dynamic variables.
The {\em activity variable}  
receives a linear combination of inputs from all other variables,
whose coefficients correspond to a weighted adjacency matrix for the network.
This combination of inputs is modulated by a
sigmoidal gain function. The {\em fatigue variable} is internal to
the neuron, and causes its activity to die down again. See Section \ref{S:RM}.

In this paper, as proof of concept, we employ rate models.
With the likely exception of Section \ref{S:LFRM} on Liapunov functions,
which relies on the specific form of rate model gain functions,
similar analyses based on alternative neuron models could be carried out in
much the same manner. Indeed, we have done this for a 
simpler network in \cite{SW23b}.

\subsection{Summary of Paper}

Section \ref{S:FFL} recalls from \cite{SW23a} the concept of
a feedforward lift network. We discuss three examples: a `toy' model to
illustrate mathematical features, and two lifts of Figure \ref{F:4CPG} 
that generate the same phase patterns as Figure \ref{F:4CPG}
and thus might equally well model biped locomotion while propagating
signals along chains of short-range connections.
We summarise three notions of stability for periodic orbits --- Liapunov, asymptotic,
and Floquet stability. We discuss `transverse' versions of these
stability conditions, appropriate to feedforward lifts, along with
a fourth form of stability: transverse stability of the synchrony subspace.
We state three basic stability theorems from \cite{SW23a}.

Section~\ref{S:RM} specialises the analysis to rate (or Wilson--Cowan) models.
A similar analysis could be performed using other neuron models, and \cite{SW23b}
suggests that stable travelling waves would be likely.
The relevant stability notions are recalled in Section \ref{S:LS}.
First, we show that the synchrony subspace of any rate model is
transversely stable. By \cite{MY60} this does not imply Floquet stability,
but it suggests that Floquet stable lifted states should be common.
In Section \ref{S:LFRM} we construct two Liapunov functions that ensure
transverse Liapunov stability for rate models when parameters
lie in specified ranges. One of them also guarantees Floquet stability
over a slightly smaller range.
We simulate the lifted periodic state numerically for four parameter sets
that satisfy the stability conditions, and compute CPG and
transverse Floquet multipliers numerically.

Section~\ref{S:ABL} applies the theory of feedforward lifts and
transverse stability to two models of biped locomotion, shown in Figure \ref{F:4CPGcascade}.
These models are lifts of the
CPG of figure \ref{F:4CPG} proposed in \cite{PG06}, itself
derived from \cite{GSBC98,GSCB99}. The lifts propagate 
phase-synchronous signals along an arbitrarily long chain with
only short-range connections. One model uses a feedforward lift to replace the
4-node CPG by an arbitrarily long feedforward chain driven by that CPG,
which might be more plausible physiologically. The second model is
{\em almost} a feedforward lift, except for lateral connections. Again, this
topology is plausible.
The synchrony subspace for the first model is transversely stable
by Theorem \ref{T:Ratestable}. We give analytic
conditions for transverse stability for the second model.
We compute transverse eigenvalues numerically along the periodic orbit.
We also compute transverse Floquet multipliers and compare these two stability notions.

Finally, we summarise the main conclusions in Section~\ref{S:C}.

\section{Feedforward Lifts}
\label{S:FFL}

%\subsection{Mathematical Definition}

A feedforward lift \cite{SW23a} of a given network is a lift with two components:
\begin{itemize}
\item[\rm (a)]
The original network, now referred to as a {\em CPG}.
\item[\rm (b)] A {\em chain}, or a more complex branching structure, which receives
signals from the CPG but does not send signals to it. Nodes in the chain
are connected in a feedforward (also called `acyclic') manner, and input arrows are arranged
so that they receive the same signals as suitable nodes in the CPG.
\end{itemize}

The formal definition of a feedforward lift is given in \cite[Definition 3.4]{SW23a}. We do not repeat it here
since the feedforward lifts in this paper are described explicitly. However, it underpins
the general stability theorems that we require.
 The main point of \cite{SW23a} is that any feedforward lift of a given network
can propagate time-periodic phase-related signals in a natural manner
along a chain or even a branching tree, and these signals are often stable
(even to synchrony-breaking perturbations). 

\begin{example}\em
\label{ex:7node_example}

Figure~\ref{F:7nodeFFZ3}
is a simple `toy' example used in \cite{SW23a} to illustrate the main ideas,
and initially we use it here for simplicity.

\begin{figure}[h!]
\centerline{%
\includegraphics[width=0.6\textwidth]{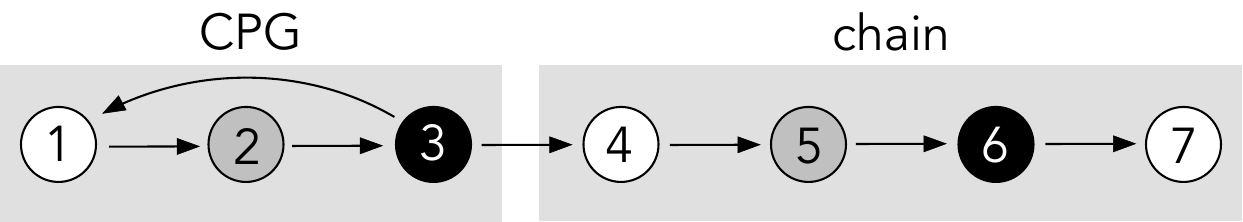}
}
\caption{A 7-node network with
one node-type and one-arrow type. The colouring (black, grey, white) is balanced.
This network naturally propagates signals such that nodes of the same colour are in synchrony, 
and successive nodes are phase-shifted by $\sot$ of the period. The CPG $\{1,2,3\}$
has cyclic group symmetry $\Z_3$.}
\label{F:7nodeFFZ3}
\end{figure}

The CPG is the subnetwork with nodes $\{1,2,3\}$
and the arrows connecting them; the chain comprises the remainder of the network.
Intuitively, nodes 1 and 4 receive the same signal from node 3 and have the same internal dynamics,
so they can synchronise. Then nodes 2 and 5 receive synchronous signals,
so they can synchronise. Similarly node 6 can synchronise with node 3, node 7
with node 4, and so on for longer chains. More formally, the colouring shown is {\em balanced}:
nodes of the same colour have inputs that are colour-isomorphic. See \cite[Chapters 8, 10]{GS23}.

The CPG $\{1,2,3\}$
has cyclic group symmetry $\Z_3$, so a natural phase pattern for a period
oscillation involves a $\sot$-period phase shift between successive nodes \cite{GS02,GS23,GRW12,GSS88,SP08}. The synchrony pattern ensures 
that these phase shifts are replicated along the chain. 
Figure \ref{F:Rate_time_series} in Section \ref{S:S7NC} shows numerical examples
of such phase patterns.
\end{example}

A more complex example of a feedforward lift is shown in Figure \ref{F:4CPGcascade} (top). The very similar topology of
Figure \ref{F:4CPGcascade} (bottom) is not a feedforward lift in the strict sense,
but more general results apply to it. We introduce these
two networks because the network of Figure~\ref{F:4CPG} need not be taken literally
as the exact CPG architecture for a biped. In particular, {\em any} lift of that
network has the same dynamics when restricted to a suitable synchrony subspace;
that is, synchronous clusters have the same waveforms and
phase relations as the nodes of the 4-node CPG. So the chains of repeating modules
in Figure \ref{F:4CPGcascade} propagate phase patterns to more distant
parts of the animal via exclusively short-range connections.
Figure \ref{F:5node} shows a non-feedforward lift, to which similar remarks
apply. However, when the lift is
feedforward  there is a systematic way to analyse stability:
\cite[Theorem 5.4]{SW23a} shows that in that case,
when transverse Floquet multipliers are stable,
the patterns concerned have the same stabilities as they do for the CPG.
(This is true for each of the main three notions of stability of a periodic state:
asymptotic, Liapunov, or Floquet.) Indeed, the lifted dynamics can
be stable even when some connections are not feedforward.

Moreover, transverse Floquet stability of the lifted periodic state
for one extra `module' in a feedforward lift
implies transverse Floquet stability for any number of extra modules.
As suggested in \cite{SW23b}, this feature may offer evolutionary advantages.

\subsection{Stability and Transverse Stability}
\label{S:STS}

As already mentioned, a crucial issue is the stability of the signals, which is necessary
for them to model reasonable biology.
In order for a propagating
signal created by a feedforward lift to be stable --- asymptotically, Liapunov, or Floquet --- the
original CPG periodic orbit must be stable (in the same sense)
to synchrony-preserving perturbations and to 
synchrony-breaking perturbations. The first condition is equivalent
to stability of the periodic orbit within the CPG state space, since
for a feedforward lift this can be identified with the synchrony subspace concerned.
The second condition can be stated as a form of {\em transverse stability},
and this condition applies unchanged to all such chains and trees. 
We summarise these ideas in Section \ref{S:SMR}.

When these stability conditions hold, a feedforward lift
propagates phase patterns of periodic states
in a robust (indeed, technically `rigid') manner, making them
suitable for locomotion models. In \cite{SW23b} we examined four standard
neuron models (FitzHugh--Nagumo, Morris--Lecar, Hindmarsh--Rose,
and Hodgkin--Huxley) from this point of view, for the `toy' 7-node network of Figure \ref{F:7nodeFFZ3},
showing numerically that
the stability conditions are often satisfied.

The present paper focuses on a fifth type of model: Wilson--Cowan
(or rate model) neurons. It has two main aims. One is to prove analytically
that such models are transversely stable, either in the sense of Liapunov stability
or more strongly Floquet stability, if certain conditions on parameters are
satisfied. The other is to study stability if propagating signals
for two feedforward lifts for a CPG for biped locomotion introduced in \cite{SW23a}.
We  use both analytic and numerical methods to do this.

\subsubsection{Three Stability Notions}
\label{S:LS}

We recall three stability notions for equilibria and periodic orbits, each stronger than the previous one. For further
information see \cite[Chapter 4]{MLS93} and \cite{L64,LL61,L92}.
Let $\|\cdot\|$ be any norm on $\R^n$.
In this paper we use the Euclidean norm unless otherwise stated.
All norms on $\R^n$ are equivalent.

\begin{definition}\em
\label{D:LAEstab}
Let $x(t)$ be an orbit of an ODE $\dot x = f(x)$. Then an
equilibrium $a_0$ is:

(a) {\em Liapunov stable} if,
for every $\eps > 0$, there exists $\delta>0$ such that if $\|x(0)-a_0\| < \delta$
then $\|x(t)-a_0\| < \eps$ for all $t > 0$. 

(b)  {\em Asymptotically stable} if it is Liapunov stable,
and in addition $\delta$ can be chosen so that $\|x(t)-a_0\| \to 0$ as $t \to +\infty$.

(c) {\em Exponentially stable} if there is a neighbourhood $V$
of $a_0$ and constants $K, \alpha > 0$ such that $\|x(t)-a_0\| < K \ee^{-\alpha t}$
for all $x(0) \in V$. 

(For some norm, not necessarily the Euclidean one, we can assume $K=1$.)
\end{definition}

Exponential stability implies asymptotic stability, which in turn implies
Liapunov stability. Neither converse is valid in general. {\em Linear stability}
(all eigenvalues of the Jacobian at $a_0$ have negative real part)
is equivalent to exponential stability \cite{HS74}.

We can transfer these stability conditions 
to a periodic orbit (trajectory) $\{a(t)\}$ via a Poincar\'e map.
An alternative, more classical, approach is to define $y(t)=x(t)-a(t)$. 
Then the  {\em system of deviations}
\begin{equation}
\label{E:SoD}
\dot y = f(y+a(t))-\dot a(t)
\end{equation}
is a non-autonomous ODE with a fixed
equilibrium $y=0$. The orbit $\{a(t)\}$ is defined to be stable, 
in any of the senses (a), (b), (c)of Definition \ref{D:LAEstab}, if
this equilibrium is stable in the same sense. 

Since $y(t) = x(t)-a(t)$ we can restate the definition for Liapunov stability of an orbit
(in particular, a periodic one):

\begin{definition}\em
\label{D:LSorbit}
The orbit $\{a(t)\}$ is {\em Liapunov stable}
if, for any $\eps > 0$, there exists $\delta >0$ such that whenever
$\|x(0)-a(0)\| < \delta$ we have $\|x(t)-a(t)\| < \eps$ for all $t \geq 0$.
\end{definition}

There are corresponding definitions of asymptotic and exponential stability
for orbits. Often these are called orbital asymptotic and exponential stability.
Again, exponential stability implies asymptotic stability, which in turn implies
Liapunov stability. Exponential stability is equivalent to Floquet stability \cite{HS74}.

\subsubsection{Liapunov Functions}

Liapunov stability can sometimes be proved using a {\em Liapunov function}
$L(x)$. 

\begin{definition}\em
\label{D:LiapFunc}
Let $a_0$ be an equilibrium of an ODE on $\R^n$. Then a (local)
{\em Liapunov function} is a map $L:N \to \R^n$, where $N$
is a neighbourhood of $a_0$, such that:

(a) $L(x) \geq 0$ for all $x$ near $a_0$.

(b) $L(x) = 0$ if and only if $x=a_0$.

\noindent
A further condition strengthens this notion:

(c) $\frac{\dd}{\dd t} L(x(t)) \leq 0$ for all $x(0) \in N, t \geq 0$, 
where $x(t)$ is the forward orbit of $x(0)$.

\noindent 
A more explicit bound implies exponential stability:

(d) There is a constant $\alpha > 0$ such that 
$\frac{\dd}{\dd t} L(x(t)) \leq -\alpha L(x(t)) $ for $t\geq0$
\end{definition}

{\em Liapunov's Second Method} \cite{L64,LL61} is the observation that the existence
of a Liapunov function satisfying (a), (b), and (c)
implies that $A$ is asymptotically stable. If only
(a) and (b) hold, it is Liapunov stable. If (d) also holds, it is exponentially stable.

\subsubsection{Transverse Stability}

For an admissible ODE of a feedforward lift, 
the three notions of stability apply directly to the CPG dynamics, ignoring
the chain that it drives. The structure of a feedforward lift implies that
the CPG dynamics can be identified naturally with (is `conjugate' to) the dynamics of
the corresponding synchronous clusters. We can therefore split the
stability conditions into two parts: stability to perturbations that preserve synchrony,
and to perturbations that break synchrony. The first part corresponds to stability
of the CPG dynamics; the second is what we have called transverse stability.

In \cite{SW23a} we define transverse asymptotic, Liapunov, and Floquet
stability for a periodic orbit. The idea is to consider the system of deviations
\eqref{E:SoD} and require the corresponding stability condition for the
fixed point defined by the periodic orbit $\{a(t)\}$ and lets us focus on
the transverse dynamics. Linearising this equation 
leads to Floquet theory. 

We prove that for each of the three types of
stability in Definition \ref{D:LAEstab}, the lifted signal $\{\tilde a(t)\}$ is stable if and only if the periodic orbit $\{a(t)\}$ 
of the CPG is stable
and it satisfies the corresponding transverse stability condition. We emphasise
that this condition depends only on the CPG dynamics, although it differs from
stability of $\{a(t)\}$ itself. This fact is a consequence of the combination of 
the feedforward structure with the imposed phase pattern.

Transverse Liapunov stability is defined in \cite[Section 5.5]{SW23a},
and a transverse version of asymptotic stability is also referred to there. It is proved in
\cite[Theorem 5.4]{SW23a} that the lifted periodic orbit $\{\tilde a(t)\}$ 
is Floquet stable if and only if $\{a(t)\}$ is Floquet stable on the CPG state space
 and $\{\tilde a(t)\}$ is transversely Floquet stable.
Analogous results hold for Liapunov and asymptotic stability. We state
the relevant theorems in Section \ref{S:SMR}.

\subsubsection{Transverse Stability of the Synchrony Subspace}

A different notion of transverse stability is discussed in \cite[Section 6]{SW23a}.
Here we consider the dynamic transverse to the synchrony subspace
at each point. Feedforward structure gives the Jacobian a block-triangular
form, with one block for the CPG and diagonal blocks for successive nodes
along the chain. The synchrony subspace is {\em transversely stable}
if these diagonal blocks have eigenvalues with negative real parts.

It is tempting to assume that this form of transverse stability should imply
transverse Floquet stability. This is true if node spaces are 1-dimensional \cite[Theorem 6.4]{SW23a}.
However, the Markus-Yamabe counterexample \cite{MY60} 
and \cite[Example 6.3]{SW23a} show that
it is false in general. Despite this, it remains a useful heuristic. We
compared transverse stability with transverse Floquet stability in \cite{SW23b}
for standard neuron models in the network of Figure \ref{F:7nodeFFZ3}.

\subsection{Stability Theorems for Feedforward Lifts}
\label{S:SMR}

We summarise the main stability results of  \cite{SW23a}, with minor changes
to conform to this paper.

The first \cite[Theorem 5.4]{SW23a} is a necessary and sufficient condition for Floquet stability
of a feedforward lift:

\begin{theorem}
\label{T:FFStab}
Let $\{\tilde a(t)\}$ be a feedforward lift of the periodic orbit $\{a(t)\}$
on the CPG. Then:

{\rm (a)} The Floquet multipliers for $\{\tilde a(t)\}$ are the
Floquet multipliers for $\{a(t)\}$ on the CPG, together with the transverse Floquet multipliers
for all nodes of the chain.

{\rm (b)} The transverse Floquet multipliers for any node of the chain are the
same as those for a synchronous node in the CPG.

{\rm (c)} $\{\tilde a(t)\}$ is stable on $P$ if and only if
$\{a(t)\}$ is stable in the CPG state space, and all transverse Floquet multipliers 
have absolute value $<1$ for every node in the CPG.
\end{theorem}

The second  \cite[Theorem 7.1]{SW23a} adds further detail for phase patterns:

\begin{theorem}
\label{T:TWtranseigen}
Assume that the CPG $\GG$ has nodes $\CC = \{1, \ldots, m\}$ 
with a cyclic automorphism group $\Z_k = \langle \alpha \rangle$,
such that $n=mk$ and $\alpha$ acts like the cycle $(1\, 2\, \ldots\, k)$
on all of its orbits on $\GG$. Let $\{a(t)\}$ be a $T$-periodic solution of an admissible ODE
with discrete rotating wave phase pattern.
Choose a set $\MM$ of orbit representatives.
Let $\widetilde{\GG}$ be obtained by lifting appropriate copies of translates 
$\MM_{k+1}, \MM_{k+2}, \ldots, \MM_l$ of this set by $\Z_k$. Then 

{\rm (a)} The periodic state $\{a(t)\}$ on $\GG$ lifts to a $T$-periodic travelling wave
state $\{\tilde{a}(t)\}$ for $\widetilde{\GG}$ with phases corresponding to the copies
$\MM_{k+1}, \MM_{k+2}, \ldots, \MM_l$ of $\MM$.

{\rm(b)}
The Floquet exponents of $\{\tilde{a}(t)\}$ (evaluated at any point) are 
those on the module $\MM$, together with the transverse
Floquet exponents for $\MM$.

{\rm(c)} If the Floquet exponents  and
the transverse Floquet exponents on $\MM$ have negative real part,
then $\{\tilde{a}(t)\}$ is stable.
\end{theorem}

The third  \cite[Theorem 5.10]{SW23a} is an analogous result for Liapunov stability:
\begin{theorem}
\label{T:FFLSstab}
The lifted periodic orbit $\{\tilde a(t)\}$ is Liapunov stable 
if and only if
$\{a(t)\}$ is Liapunov stable in the CPG state space 
and transversely Liapunov stable.
\end{theorem}

There is a natural analogue of Theorem~\ref{T:TWtranseigen}(c) for Liapunov
stability, proved in the same manner. We do not state it here.

\begin{remark}\em
If network $\GG$ is a lift of network $\HH$ then {\em all}
dynamic states for $\HH$ are preserved (subject to the appropriate
synchrony pattern) when lifted to $\GG$,
not just equilibria and periodic states. However, the concept of
stability is more complicated for more
complex states such as chaos. Indeed, stability on the CPG alone
alone is more complicated, because there are many different
concepts of an `attractor' \cite{M85}. Transverse stability opens
up further complications, and is mainly
understood for discrete dynamics and only one transverse dimension
\cite{ABS94,ABS96}, and even then there are several different scenarios.
Synchronous chaos is important as one method to ensure secure communications,
and has been widely studied from an electronic engineering viewpoint,
see for example \cite{Pec08,PB05,PC98}.
\end{remark}

\section{Rate Models}
\label{S:RM}

Having motivated the notion of a feedforward lift, and 
set up the theoretical aspects that we require, we now discuss
stability properties of lifted periodic orbits in three specific
feedforward lifts: the 7-node network of Figure \ref{F:7nodeFFZ3} and the two
proposed networks for biped locomotion in Figure \ref{F:4CPG}.

We model the dynamics of each node using a {\em rate model}, 
also called a {\em Wilson--Cowan} model \cite{ET10}. 
Here each node represents a neuron (or a population of neurons), and
the node variables correspond to the rate at which the neuron fires.
The state of node $\NN_{i}$ is described by a two-variable vector
$x_{i} = (x^E_{i},x^H_{i})$, where $x^E_{i} \in \R$ is an {\em activity variable}
and  $x^H_{i} \in \R$ is a {\em fatigue variable}. 

Nodes are coupled through a {\em gain function} $\GG: \R \rightarrow \R$. 
This satisfies the following conditions:
\begin{itemize}
\item[\rm (1)] 
$0 \leq \GG(x) \leq \rho$ for all $x \in \R$ and for some $0 <\rho \in \R$.
\item[\rm (2)]
$0 \leq \GG'(x) \leq \sigma$ for all $x \in \R$ and for some $0 <\sigma \in \R$.
\item[\rm(3)]
$\lim_{x\to -\infty} \GG(x) = 0$ and $\lim_{x\to +\infty} \GG(x) = \rho$.
\item[\rm (4)]
$\GG(x) \approx 0$ when $x \leq 0$.
\end{itemize}
Usually $\GG$ is sigmoid, so there is a unique inflexion point 
$x_0$ where $\GG''(x_0) = 0$, but we do not make this a requirement.

A standard choice of gain function is
\begin{equation}
\label{E:Ggen}
\GG(x) = \frac{a}{1 + \ee^{-b (x - c)}}
\end{equation}
where $a, b, c > 0$. (This notation is employed here for consitency
with the literature and should not be confused
with other uses in this paper.) See \cite{LC02,SCRR07} for a 
discussion of such models in the context of binocular rivalry.
 With reasonable choices of $a, b, c$
the value of $\GG(x)$ is very small when $x \leq 0$. 
A standard choice in the literature is $a = 0.8, b = 7.2, c = 0.9$,
for which $k = 0.8$. In the simulations of Section~\ref{S:Sim} we
use similar but simpler parameters.
Note that although $\GG(x)$ is very small for negative $x$, it is not identically
zero. Thus a `passive' node might not have a state exactly equal
to zero; instead, its values will be very close to zero. This assumption
is convenient for programming simulations and is standard in the literature.

The node equations are specified in terms of a {\em connection matrix}
\[
A = (\alpha_{ij}) \quad 1 \leq i, j \leq n
\]
where $n$ is the number of nodes, and $\alpha_{ij}$ is the
strength of the connection from node $j$ to node $i$ (positive
for excitation, negative for inhibition). The equations take the form
\begin{equation}
\label{E:genrateEq}
\begin{array}{rcl}
\varepsilon \dot{x}^E_i & = & -x^E_i + \GG\left( -gx^H_i +  \sum_{j \neq i} \alpha_{ij} x^E_j + I_i \right) \\
\dot{x}^H_i & = & x^E_i-x^H_i 
\end{array}
\end{equation}
for $1 \leq i \leq n$. 
This is a fast-slow system, 
and $\varepsilon \ll 1$ is the timescale on which $x^E_{i}$ evolves.
The parameter $g > 0$ is the strength of reduction of the
activity variable by the fatigue variable. 
$I_{i}$ is the external signal strength to node $i$. For simplicity it is standard to let
$I_{i} = 0$ if node $i$ is inactive, and $I_{i} = I$ if node $i$ is active, for a
constant $I$ usually taken to equal $1$.
However, we consider the general case here.

\subsection{Transverse Stability for Rate Models}
\label{S:TSTRM}

Rate models have 2-dimensional node spaces, so 
\cite[Theorem 6.4]{SW23a} does not apply. Nonetheless, we can
state and prove several theorems giving conditions under which
lifted periodic states for a rate model are transversely stable,
transversely Liapunov-stable, or transversely Floquet-stable.

The simplest (but least useful) case is transverse stability.
We prove that all rate models are transversely stable. As noted at the end of Section \ref{S:STS},
this does not imply Floquet-stability of lifted periodic
orbits, but it suggests that Floquet-stability is likely.

The following lemma is well known and easy to prove:
\begin{lemma}
\label{L:2x2}
All eigenvalues of a $2 \times 2$ matrix $A$ have negative real part if
and only if $\trace(A) < 0$ and $\det(A) > 0$.
\qed\end{lemma}

We immediately deduce:

\begin{theorem}
\label{T:Ratestable}
For any rate model and any feedforward lift, the synchrony 
space $S$ is transversely stable at every point.
\end{theorem}
\begin{proof}
The diagonal part of the Jacobian, for node $i$, is equal to
\begin{equation}
\label{E:rateJac}
A = \Matrix{-1/\eps & -\frac{g}{\varepsilon}\GG^\prime\left( -gx^H_i +  \sum_{j \neq i} \alpha_{ij} x^E_j + I_i \right) \\
1 & -1}
\end{equation}
Therefore
\beqn
\trace(A) &=& -1/\eps -1 < 0 \\
\det(A) &=& 1/\eps + \frac{g}{\varepsilon}\GG^\prime\left( -gx^H_i +  \sum_{j \neq i} \alpha_{ij} x^E_j + I_i \right) > 0
\eeqn
because $\eps > 0$ and $\GG'(x) \geq 0$ for all $x$. Now apply Lemma~\ref{L:2x2}.
\end{proof}

For this reason we do not show plots of the real parts of transverse eigenvalues
for the rate models considered below.

\subsection{Liapunov Functions for Rate Models}
\label{S:LFRM}

In this section we construct two Liapunov functions for rate models, 
which show that under certain conditions on parameters,
lifted periodic states are transversely Liapunov stable. By condition (c)
of Definition \ref{D:LiapFunc}, the same Liapunov function
implies transverse asymptotic stability. With some
extra effort we can use condition (d) of Definition \ref{D:LiapFunc}
to parlay this result into a proof of transverse Floquet-stability under
slightly different conditions. It is probably possible to improve 
on these results. They can be viewed as proof-of-concept;
we apply them to a model of biped locomotion in Section~\ref{S:PCA}.

To simplify expressions, scale time from $t$ to $s= \eps t$. Then $\dd s = \eps \dd t$ 
so $\dd/\dd s = \frac{1}{\eps}\dd/\dd t$.
In scaled time (now called $t$) we can
write the Floquet equation for transverse stability as
\[
\Matrix{\dot u \\ \dot v} = \Matrix{-1 & -g \eta(t) \\ \eps & -\eps}\Matrix{ u \\ v}
\]
where $\eta(t)$ is the input to the first node $c$ of the chain from the CPG:
\begin{equation}
\label{E:eta(t)}
\eta(t)= \GG^\prime\left( -gx^H_{[c]}(t) +  \sum_{j \neq i} \alpha_{ij} x^E_{[j]}(t) + I_i \right) \geq 0
\qquad \mbox{for all}\ t \in \R
\end{equation}
Here $[c]$ is the node in the CPG that is synchronous with node $c$.

In fact, $\eta(t) > 0$ for all $t \in \R$ for the gain function \eqref{E:Ggen}.
Since a periodic orbit is compact, there exist bounds $D_0, D$ such that
\begin{equation}
\label{E:gamma_bounds}
0 \leq  D_0 \leq \eta(t) \leq D
\end{equation}
 and $g, \eps \geq 0$. For the gain function \eqref{E:Ggen} we have $D_0 > 0$.
Define 
\[
\zeta(t) = g\eta(t) \geq 0
\]
and when $t$ is irrelevant write $\zeta$ in place of $\zeta(t)$.
Then
\[
0 \leq \zeta(t) \leq gD
\]

\subsubsection{Liapunov Function 1}

The function $\sqrt{\eps x^2 + y^2}$ is a norm on $\R^2$. We claim that
\[
L(t) = \eps u^2 + v^2
\]
is a Liapunov function provided $-1 < \zeta < 3$.

\begin{remark}\em
The norm itself is also a Liapunov function under the same conditions. 
Indeed,
\[
\frac{\dd}{\dd t} \sqrt{\eps u(t)^2 + v(t)^2} 
	= \frac{\frac{\dd}{\dd t}(\eps u(t)^2 + v(t)^2)}{2\sqrt{\eps u(t)^2 + v(t)^2}}
\]
so this norm of $(u(t),v(t))$ is a Liapunov function provided $L(t)$ is.
We work with $L(t)$ for simplicity.
\end{remark}

We have:
\beqn
\dot L(t) &=& 2 \eps u \dot u + 2 v \dot v\\
&=& 2 \eps u (-u- \zeta v) + 2 v \eps (u-v) 
\eeqn
so
\beqn
-\shf \dot L(t) &=& 
\eps u (u+ \zeta v) +  v \eps (-u+v) \\
&=& \eps u^2 + (\zeta\eps -\eps)uv +\eps v^2\\
&=& \eps(u^2+(\zeta-1)uv+v^2)
\eeqn
We seek regions in which $u^2+(\zeta-1)uv+v^2 \geq 0$. Then $\dot L(t) \leq 0$
so $L(t)$ is a Liapunov function.

To make the quadratic form $Q(u,v) = u^2+(\zeta-1)uv+v^2$ positive definite,
consider its discriminant
\[
\Delta = (\zeta-1)^2 - 4 = \zeta^2-2\zeta-3
\]
This is negative when $-1 < \zeta < 3$. Now $\Delta < 0$ implies that
$Q$ is definite, and it is clearly positive definite (since when $\zeta=1$, inside that range,
it is $u^2+v^2$).

This proves:
\begin{proposition}
\label{P:Liap1}
If $0 \leq gD \leq 3$ then the lifted periodic orbit $\{\tilde a(t)\}$ is transversely Liapunov stable. 
\qed \end{proposition}

With a little more effort we can prove transverse Floquet stability:

\begin{proposition}
\label{P:trans_stab_Floq}
If $\eps \leq 4$ and $0 \leq gD_0 <  \shf(2+\sqrt{3(4-\eps)})$,
the lifted periodic orbit $\{\tilde a(t)\}$ is transversely Floquet 
stable.
%\RED{\em\ FOOTNOTE}\footnote{Can't we take
%$B_0$ as small as we like? Then we get a completely general result for all $\eps < 4$.}
\end{proposition}
\begin{proof}
We use condition (d) in Section \ref{S:LS}.
Suppose we can obtain a more specific estimate of the form
\[
\dot L(t) \leq -k L(t) < 0 \qquad \mbox{or equivalently,}\qquad \dot L(t) + k L(t) < 0
\]
for some $k > 0$. Then we deduce that $L(t) \leq K\ee^{-kt}$ for a constant $K$,
which tends exponentially to 0. This proves transverse  exponential stability, 
which is equivalent to Floquet stability.

In particular, this condition holds if $k = gD_0$ for $D_0$ as in \eqref{E:gamma_bounds}.
(We need $D_0$, not $D$, because of the minus sign. Recall from \eqref{E:eta(t)} that
$\eta(t) \geq 0$ for all $t \in \R$.)

To obtain such a constant $k$, consider
\beqn
\dot L(t) + kL(t) &=& -2\eps(u^2+(\zeta-1)uv+v^2)+k(\eps u^2+v^2)\\
	&=& u^2(-2\eps+k\eps)+2 uv\,\eps(\zeta-1)+v^2(-2\eps+k)
\eeqn
Let $k = \eps/2$. Then
\beqn
\dot L(t) + kL(t) &=& u^2(-2\eps+\eps^2/2)+2uv\,\eps(\zeta-1)+v^2(-\frac{3}{2}\eps)\\
	&=& -[u^2(2\eps-\eps^2/2)+2uv\,\eps(1-\zeta)+v^2(\frac{3}{2}\eps)]
\eeqn
We seek to make the quadratic form 
\[
Q(u,v)=u^2(2\eps-\eps^2/2)+2 uv\,\eps(1-\zeta)+v^2(\frac{3}{2}\eps)
\]
positive definite. Then the minus sign makes $\dot L(t) + kL(t) \leq 0$ as required.

The discriminant of Q is
\[
\Delta = 4\eps^2(1-\zeta)^2-4 (2\eps-\eps^2/2)(\frac{3}{2}\eps) = \eps^2(4\zeta^2-8\zeta-8+3\eps)
\]
For arbitrary $\eps >0$
the zeros of $\Delta$ are $\shf(2\pm\sqrt{3(4-\eps)})$. These are real
whenever $\eps \leq 4$. When $\eps$ lies between these zeros, $\Delta < 0$.
\end{proof}

\begin{example}\em
\label{ex:Floq_ex}
When $\eps = 0.5$ , $\Delta$ is negative
whenever $-0.224745 \leq \zeta \leq 2.22474$. With  $\eps = 0.6$ it is negative
whenever $-0.596872 \leq \zeta \leq 2.59687$.

As $\eps \to 0$ the bounds $\shf(2\pm\sqrt{3(4-\eps)})$
tend to  $1\pm\sqrt{3} =  2.73205, -.73205$.
\END\end{example}

\subsubsection{Liapunov Function 2}
The function $\sqrt{\eps^2 x^2 + y^2}$ is also a norm on $\R^2$. Let
\[
L(t) = \eps^2 u(t)^2 + v(t)^2
\]
Then
\[
-\sq \dot L(t) = \eps^2 u^2 + (\zeta\eps^2-\eps)uv + \eps v^2 = \eps(\eps u^2 + (\zeta\eps-1)uv +  v^2)
\]
If $\eps u^2 + (\zeta\eps-1)uv +  v^2 \geq 0$, then $\dot L(t) \leq 0$
so $L(t)$ is a Liapunov function. We seek regions in which this condition holds.

To make the quadratic form $Q(u,v) = \eps u^2 + (\zeta\eps-1)uv +  v^2$ positive definite,
consider its discriminant
\[
\Delta = \eps^2[(\eps \zeta-1)^2 - 4\eps ]
\]
Again $\Delta < 0$ implies that
$Q$ is definite, and it is clearly positive definite. This happens when
\[
\frac{1-2\eps}{\eps} < \zeta < \frac{1+2\eps}{\eps} 
\]

If, for example, $\eps = 0.5$,
this is negative when $-0.828427\leq \zeta \leq 4.82843$. 

If, for example, $\eps = 0.6$,
this is negative when $-0.596872\leq \zeta \leq 2.59687$. 

This proves:
\begin{proposition}
\label{P:Liap2}
If $0 \leq g D \leq 4.82843$ and $\eps = 0.5$, then the lifted periodic orbit
is transversely Liapunov stable for $t \geq 0$.
\qed \end{proposition}

Estimates like this can be obtained for a variety of similar functions $L(t)$.
Since the above results are `proof of concept', we do not seek more
general statements.
%%For example, $L(t) = \eps^2 u(t)^2 + v(t)^2$ is another Liapunov function
%for suitable conditions on parameters.

%\RED{ FOOTNOTE}\footnote{We should be able to parlay this into Floquet-stability by a similar
%calculation, but I've not yet made that work.}

\subsection{Simulations for 7-Node Chain}
\label{S:S7NC}

Propositions~\ref{P:Liap1} and \ref{P:trans_stab_Floq}
apply only to {\em transverse} stability (in the Liapunov or Floquet sense).
For stability of the lifted periodic state we also require the CPG
periodic state to be stable. We currently have no analytic results about this issue, 
but we present some numerical computations of
both CPG and transverse Floquet multipliers.

In this subsection we investigate a rate model for the 7-node chain, of the form
\begin{equation}
\label{E:7rateEq}
\begin{array}{rcl}
\varepsilon \dot{x}^E_i & = & -x^E_i + \GG\left( -gx^H_i +  \alpha x^E_{c(i)} + I \right) \\
\dot{x}^H_i & = & x^E_i-x^H_i 
\end{array}
\end{equation}
for $1 \leq i \leq 7$. Here $c(i) = i-1$ when $i > 1$, and $c(1)=3$.
We use the gain function 
\begin{equation}
\label{E:gainfunc}
\GG(x) = \frac{1}{1+\ee^{-8(x-1)}} = \frac{1}{1+\ee^{8(1-x)}} 
\end{equation}
from \cite[Section 14]{S14}.

For a wide range of parameters  leading to the travelling wave with 
$\sot$-period phase shifts,
the computed transverse Floquet exponents indicate stability. 
Table \ref{T:FloqMultRATE} shows CPG and transverse Floquet multipliers
for the 7-node chain with 3-node CPG, for the four parameter sets
listed in (\ref{E:RATEparams1}--\ref{E:RATEparams4}).
As is well known \cite[Section 1.5]{GH83}, the Floquet operator always has at least one multiplier
equal to $1$, corresponding to an eigenvector tangent to the orbit.
This multiplier occurs on the CPG, but not among the transverse multipliers.

\begin{table}[!htb]
\small
\begin{center}
\begin{tabular}{|l|l|l|l|l|l|}
\hline
param. & $T$ & CPG multipliers  & abs. & transverse multipliers & abs. \\
\hline
\hline
\eqref{E:RATEparams1} & $5.783$ & $1$ 
	& $1$ & $0.546$ & $0.546$\\		
& & $0.433 \pm 0.181 \,\ii$ & $0.470$ &0.00315& $0.00315$ \\
stable & & $0.396 $ & 0.396 & &  \\
& & $0.0368 $ & 0.0368 & &  \\
& & $1.58\times 10^{-6}$  & $1.58\times 10^{-6}$  & & \\
\hline
\eqref{E:RATEparams2}& $4.612$& $1$ & $1$ & $0.0898$ & $0.0898$\\
& & $0.0326 \pm 0.0934 \,\ii$ & $0.0989$ & 0.0110&0.0110 \\
stable& & $0.021 \pm 0.0372\,\ii$ & $0.0428$ & & \\
& & $0.0000538 $ & $0.0000538$ & & \\
\hline
\eqref{E:RATEparams3}& $3.146$ & $1$ & $1$ & $0.0199 \pm 0.150 \,\ii$ & $0.151$\\
& & $0.484 \pm 0.207 \,\ii$ & $0.526$ & & \\
stable& & $ 0.419$ & $ 0.419$ & & \\
& & $0.0578$ & $0.0578$ & & \\
& & $0.00178$ & $0.00178$ & & \\
\hline
\eqref{E:RATEparams4}&  $4.373$ & $1$ & $1$ & $0.0260$& $0.0260$ \\
&  & $ -0.0120 \pm 0.0269 \,\ii $ & $0.0294$ & 0.0146& 0.0146\\
stable& & $0.0138$ & $0.0138$ & & \\
& & $-0.000996 \pm 0.00190 \,\ii$ & $0.00215$ & & \\
\hline
\end{tabular}
\caption{CPG and transverse Floquet multipliers for the four rate model examples.
The CPG periodic orbit and the lifted periodic orbit are stable for all four parameter sets.
All entries stated to three significant figures
after the decimal point.}
\label{T:FloqMultRATE}
\end{center}
\end{table}

\begin{figure}[h!]
\centerline{%
\includegraphics[width=0.4\textwidth]{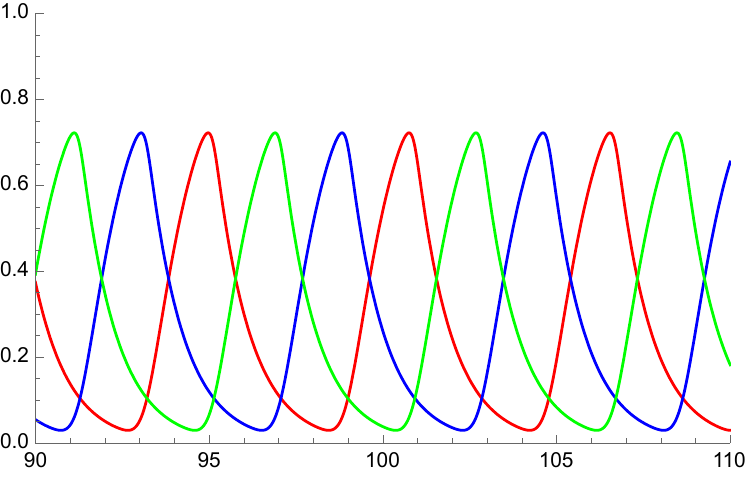}\qquad\
\includegraphics[width=0.4\textwidth]{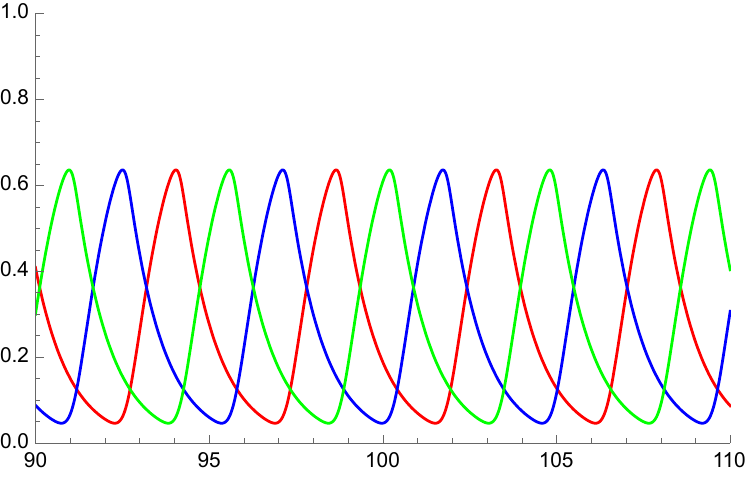}
}
\centerline{%
\includegraphics[width=0.4\textwidth]{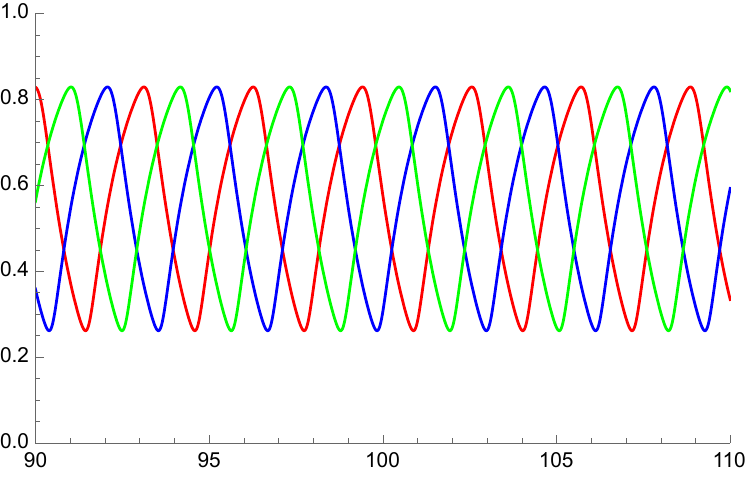}\qquad\
\includegraphics[width=0.4\textwidth]{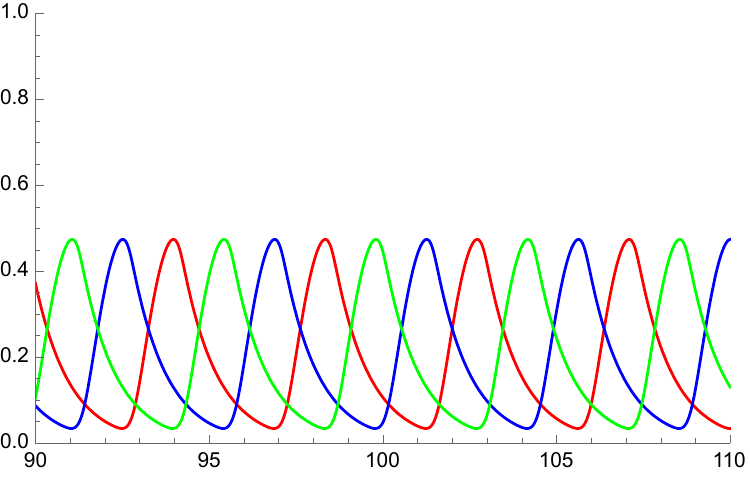}
}
\caption{Simulated time series for $x_1^E,x_2^E,x_3^E$ for the 7-node chain.
Parameter sets are (top left) \eqref{E:RATEparams1}, (top right) \eqref{E:RATEparams2}, 
(bottom left) \eqref{E:RATEparams3}, (bottom right) \eqref{E:RATEparams4}.}
\label{F:Rate_time_series}
\end{figure}

Figure~\ref{F:Rate_time_series} shows the corresponding time series
for the three different synchrony classes.
We consider four typical parameter sets:
\begin{eqnarray}
 \label{E:RATEparams1}
&& I=2 \quad \alpha = -5 \quad  g=2 \quad \eps = .1\\
\label{E:RATEparams2}
&& I=2 \quad \alpha = -5 \quad  g=2 \quad \eps = .5\\
\label{E:RATEparams3}
&& I=4 \quad \alpha = -3 \quad  g=2 \quad \eps = .2\\
\label{E:RATEparams4}
&& I=2 \quad \alpha = -8 \quad  g=3 \quad \eps = .8
\end{eqnarray}

\begin{remark}\em
Conditions for a given primary biped gait of the 4-node CPG to be the `first bifurcation', hence
potentially Floquet-stable, are given in \cite{S14}. Because the symmetry group
$\Z_2\times \Z_2$ is abelian, the isotropy subgroup acts either trivially (for the hop)
or as $\Z_2$ (run, jump, walk) when the kernel of the action is factored out. 
Therefore the usual stability criterion for generic Hopf
bifurcation applies: a supercritical branch is stable and a subcritical branch is unstable.
See \cite[Chapter 1 Section 2 Theorem III]{HKW81}.

The simulations clearly indicate that the Hopf branches are supercritical near
bifurcation, because the phase patterns do not depend on randomly chosen
initial conditions. It would be of interest to confirm this analytically, which should
be a feasible calculation.
\end{remark}

\section{Application to Biped Locomotion}
\label{S:ABL}

We now introduce a
feedforward lift and a modified version of it, slightly relaxing the 
feedforward condition, which provides a reasonable model for
locomotion of animals with many segments, such as centipedes, millipedes,
worms, snakes, and lampreys. Something similar may happen in the spinal column
of mammals and reptiles.

We start with the 4-node network of Figure~\ref{F:4CPG},
proposed in \cite{PG06} as a CPG for bipedal locomotion.
Each leg corresponds to two nodes: one flexor and one extensor.
Figure \ref{F:4CPGcascade} shows two possible lifts, for the balanced
colouring shown in shades of white and grey.
The top one is a feedforward lift. The second is not totally feedforward
because of the lateral connections, such as those between nodes 5 and 7, or 6 and 8,
but similar ideas apply with minor
modifications.

Theorem \ref{T:TWtranseigen} applies directly to the top network.
Therefore, if the periodic orbit of the CPG is Floquet-stable and the lifted orbit
is transversely Floquet-stable, the lifted periodic orbit is Floquet-stable.
Figure \ref{F:4CPGcascade} (bottom) is perhaps more plausible biologically, because
the extra nodes in the cascade are grouped
into modules with the same architecture, with
nearest-neighbour coupling between modules.
Each module comprises two nodes
with bidirectional connections; such connections are said to be {\em lateral}. 
Examples of such modules are $\{5,7\}$ and $\{6,8\}$.
The CPG has much the same architecture
as each module, except that the CPG has four arrows pointing towards 
the left (upstream). The structure can be continued arbitrarily far, providing
a simple way to propagate the oscillation patterns. (The total number
of columns can be odd or even.) It would be easy for this kind of network to evolve
in an organism such as an arthropod, from which most modern
vertebrates evolved.

Figure \ref{F:4CPGcascade} (bottom) is reminiscent of
the network in the nervous system that controls the
heartbeat of the medicinal leech, Figure \ref{F:leech}.
It is very much in the spirit of experimentally observed CPGs
\cite{TBBFPK13,G03,K06,G09,GJ09}.
It is also similar to the models studied in 
\cite{K88, KE86, KE88, KE90}, but these 
papers mainly look at continuum limits.

Apart from the shaded square of four nodes (which reproduce Figure~\ref{F:4CPG})
no arrows point towards the left, so it is {\em almost} a feedforward network.
However, because there are bidirectional {\em lateral} connections 
(dashed arrows, vertical in the figure), Theorem~\ref{T:Ratestable}
does not apply directly. Since these lateral connections are
biologically plausible we consider this generalisation
in its own right, using similar methods.
We will show that transverse stability of the synchrony subspace
 can occur in a rate model for this network,
but it is not valid for all coupling strengths. Transverse Liapunov and Floquet stability
also occur for suitable parameters.

\subsection{Rate Equations for 4-node CPG}
\label{S:RE4CPG}
First, we consider the 4-node CPG of Figure \ref{F:4CPG}.
Following \cite{S14},
explicit rate equations  for this CPG
(written in a convenient order in which the activity variables and fatigue variables are collected together) are:
\begin{equation}
\label{E:4noderateeq}
\begin{array}{rcl}
	\varepsilon \dot{x}_1^E &=& -x_1^E + \GG(I - g x^H_1 + \alpha x^E_4 + \beta x^E_3 + \gamma x^E_2) \\
	\varepsilon \dot{x}^E_2 &=& -x^E_2 + \GG(I - g x^H_2 + \alpha x^E_3 + \beta x^E_4 + \gamma x^E_1) \\
	\varepsilon \dot{x}^E_3 &=& -x^E_3 + \GG(I - g x^H_3 + \alpha x^E_2 + \beta x^E_1 + \gamma x^E_4) \\ 	 
	\varepsilon \dot{x}^E_4 &=& -x^E_4 + \GG(I - g x^H_4 + \alpha x^E_1 + \beta x^E_2 + \gamma x^E_3) \\
	\dot{x}^H_1 &=& x^E_1 - x^H_1 \\
	\dot{x}^H_2 &=& x^E_2 - x^H_2 \\
	\dot{x}^H_3 &=& x^E_3 - x^H_3 \\			
	\dot{x}^H_4 &=& x^E_4 - x^H_4 
\end{array}
\end{equation}
Here the parameters are:
\begin{eqnarray*}
\alpha &=& \mbox{strength of {\em diagonal} connection} \\
\gamma &=& \mbox{strength of {\em medial} connection} \\
\beta &=& \mbox{strength of {\em lateral} connection} \\
\varepsilon &=& \mbox{fast/slow dynamic timescale} \\
g &=& \mbox{strength of reduction of activity variable by fatigue variable} \\
I &=& \mbox{input}
\end{eqnarray*}
See Figure \ref{F:4CPG}.

The dynamics and bifurcations of \eqref{E:4noderateeq} are studied
in \cite{S14}. Despite the nonlinear form of the equations, 
much of the bifurcation behaviour can be derived analytically. 
The only local bifurcation from a fully synchronous steady state
that can lead to stable states
near the bifurcation point is the `first bifurcation'. Here either one real eigenvalue 
passes through zero while the others have negative real part, or
a complex conjugate pair of eigenvalues crosses the imaginary axis
while the others have negative real part. These bifurcations are, respectively,
steady-state (leading to a branch of equilibria) and Hopf (leading to a branch
of periodic states. It turns out to be possible to find explicit conditions
on the connection strengths $\alpha, \beta, \gamma$ that ensure that
the first local bifurcation leads either to
one of four symmetry types of steady-state bifurcation, or
one of the four symmetry types of Hopf bifurcation---the four primary gaits.
It is also possible to identify explicitly
the region of $(\alpha, \beta, \gamma)$-space in which no local bifurcations occur.
Here the synchronous steady state remains stable and changes continuously
whatever value $I$ takes. This region corresponds to small values
of the connection strengths.

The argument
of the gain function in \eqref{E:4noderateeq} is a linear combination of state variables.
This implies that the four regions of $(\alpha, \beta, \gamma)$-space
corresponding to primary gaits are connected polyhedra.
Each is a truncated pyramid
with an equilateral triangle as base. They are congruent,
and together they form a hollow
regular tetrahedron centred at the origin, Figure~\ref{F:tetra}(left). This
is the complement of a small tetrahedron $\TT_\ast$ in a larger
tetrahedron $\TT^\ast$: formulas for 
the coordinates of their vertices can be found in \cite{S14}.

The four primary gait regions are relatively large when $K$ is significantly greater than $k$.
This is the case for reasonable values of the connection strengths
and other parameters in the model, 
so the structure provides a robust way to select specific gait patterns.

\begin{figure}[htb]
\centerline{%
\includegraphics[height=1.8in]{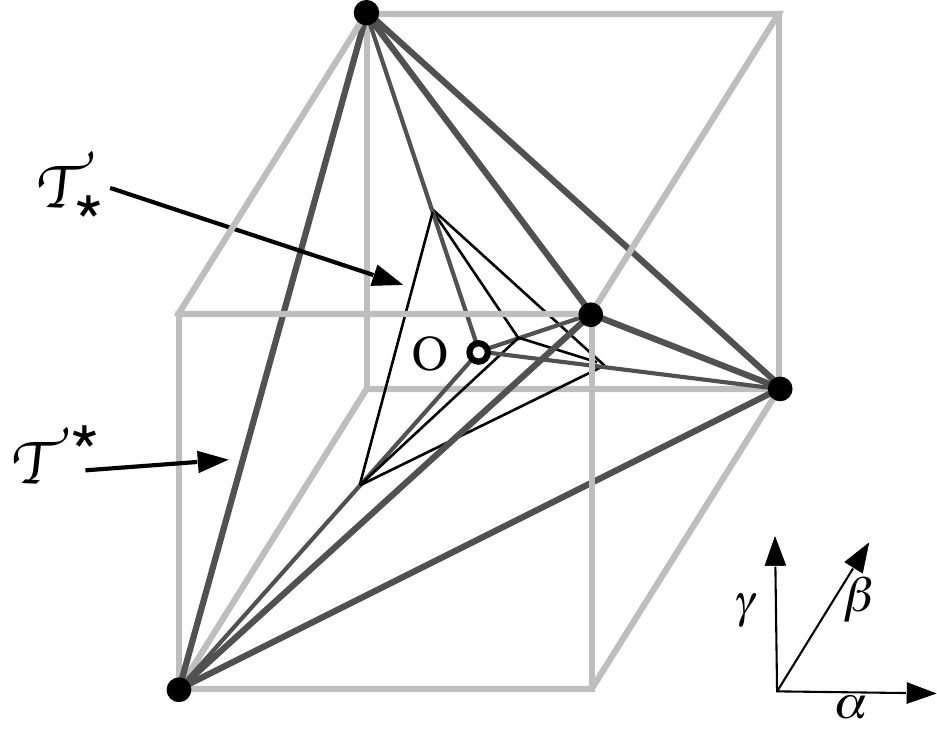} \quad 
\includegraphics[height=1.8in]{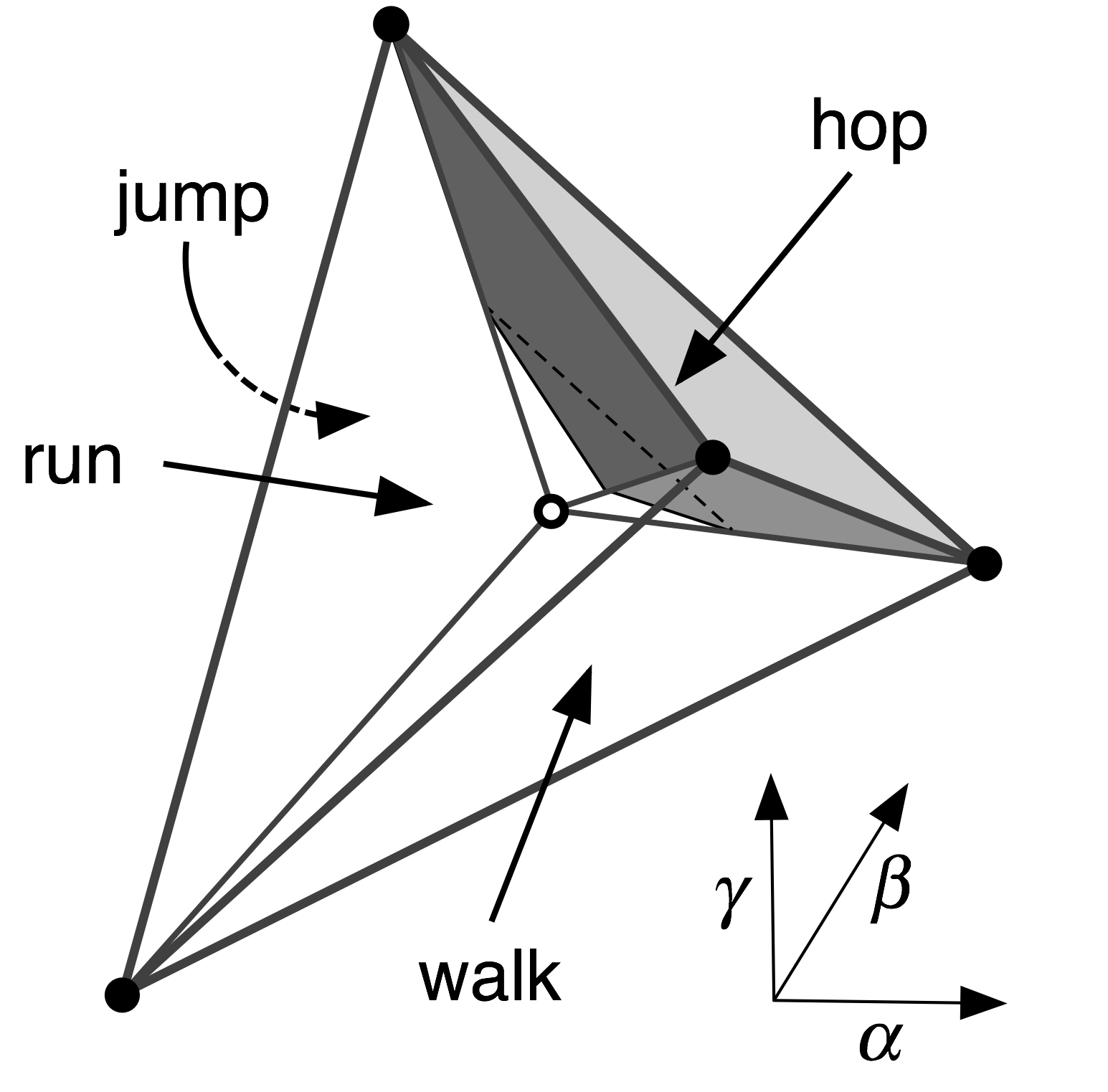} \quad
\includegraphics[height=2in]{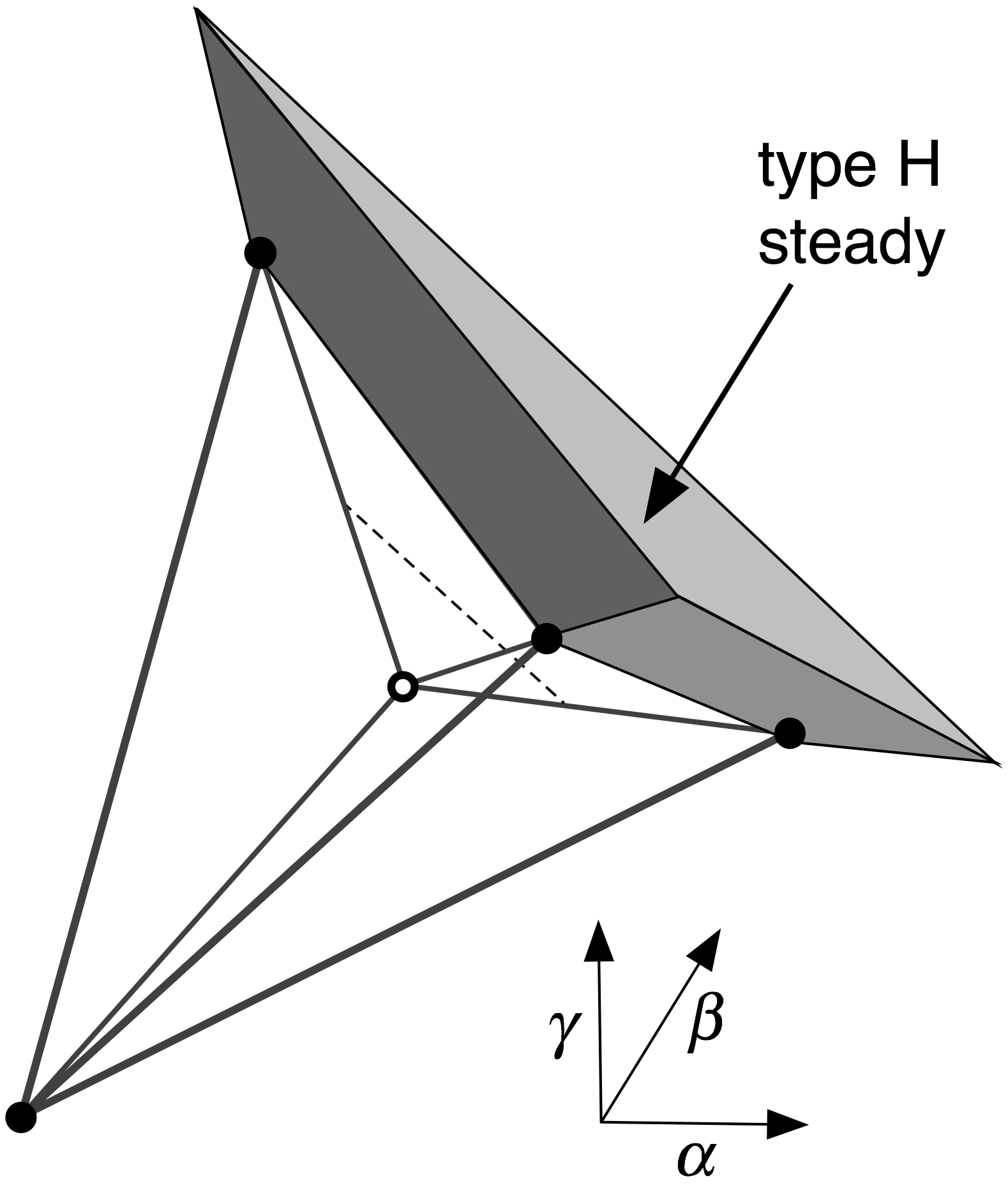}
}
\caption{Tetrahedral partition of  $(\alpha, \beta, \gamma)$-space. {\em Left}: The two
tetrahedra $\TT^\ast, \TT_\ast$. {\em Middle}: Region leading to
the hop gait. Base triangle left unshaded to show interior.
{\em Right}: Extending the pyramid to define
the corresponding steady-state regions, which extend to infinity
in the direction indicated.}
\label{F:tetra}
\end{figure}

Diekman \etal~\cite{DGMW12} analyse rate models in a network 
with the same symmetry group $\Z_2 \times \Z_2$ 
in a model of binocular rivalry. They assume a more general form of the gain
function, but take $\beta = 0, \gamma < 0$
because some connections in their model are always inhibitory.

\subsection{Simulations}
\label{S:Sim}

The four primary biped gaits in Table \ref{T:primary_gaits}
are observed numerically in the rate model \eqref{E:4noderateeq}; see
\cite[Section 14]{S14}. The gain function used there is \eqref{E:Ggen}. 
For simplicity, that paper assumes values
$a = 1, b = 8, c = 1$ in simulations, and we do the same here for
consistency.  Thus we take
\begin{equation}
\label{E:standardgain}
\GG(x) = \frac{1}{1+\ee^{-8(x-1)}} = \frac{1}{1+\ee^{8(1-x)}} 
\end{equation}
We now consider the corresponding lifted 
states for the two networks of Figure \ref{F:4CPGcascade}.
In particular, we discuss transverse Floquet multipliers
for parameter values that yield each primary gait.

\subsubsection{Synchrony and Phase Patterns}

We simulated the networks of Figure \ref{F:4CPGcascade} 
for a single additional module, 
using the parameter values of \cite[Section 14]{S14}
used to obtain Figures 19 (hop), 20 (run), 21 (jump), and 22 (walk) of that paper.
In all cases the results were identical to those of the CPG alone, with the
synchrony or phase pattern propagating correctly to the second module.

As explained in \cite[Section 4.4]{SW23a}, if the lifted periodic state is stable for one extra module
then it is stable for an arbitrary number of extra modules. Since the periodic oscillations 
observed in numerical experiments, and their synchrony patterns, are robust to changes in initial conditions,
we expect them to be stable. Thus the simulation for one extra module
indicates stability for any number of extra modules. The computed
Floquet multipliers in Tables~\ref{T:FloqBIPED} and \ref{T:FloqBIPED2module} confirm this
for the parameters stated there.

\subsection{Transverse Eigenvalues}

We now consider the transverse eigenvalues for the two networks in Figure~\ref{F:4CPGcascade}.

\subsubsection{Transverse Eigenvalues, $1$-Node Module}

Figure~\ref{F:4CPGcascade} (top) is a genuine feedforward lift. By Theorem~\ref{T:Ratestable},
all transverse eigenvalues have negative real parts. The same
caveat concerning implications for  Floquet stability applies.

\subsubsection{Transverse Eigenvalues, $2$-Node Module}

Theorem~\ref{T:Ratestable} does not apply to Figure~\ref{F:4CPGcascade} (bottom).
As observed earlier, it is enough to consider
the transverse eigenvalues for the module $\{5,7\}$. 

This network is not a feedforward lift according to \cite[Definition 3.4]{SW23a}
because $\widetilde\GG$ contains length-2 loops (the lateral connections)
that are not in $\GG$. However, a similar result
holds provided we replace the diagonal block by a block corresponding
to a {\em pair} of laterally connected nodes. For this calculation number those
nodes as 1, 2. Note that each node state space is 4-dimensional, so we are working
with a $4 \times 4$ matrix.

The equations for these two nodes are:
\beqn
\dot{x}^E_1 & = & -\frac{1}{\varepsilon }x^E_1 + \frac{1}{\varepsilon } \GG\left( -gx^H_1 +  \sum_{j\neq i} \alpha_{1j} x^E_j + I_1 \right) \\
\dot{x}^E_2 & = & -\frac{1}{\varepsilon }x^E_2 + \frac{1}{\varepsilon }\GG\left( -gx^H_2 +  \sum_{j\neq i} \alpha_{2j} x^E_j + I_2 \right) \\
\dot{x}^H_1 & = & x^E_1-x^H_1 \\
\dot{x}^H_2 & = & x^E_2-x^H_2
\eeqn
where we have grouped activity variables and fatigue variables together,
which simplifies the calculation of eigenvalues.

The Jacobian is:
\[
J = \Matrix{-\frac{1}{\varepsilon} & \frac{h}{\varepsilon}G_1 & -\frac{g}{\varepsilon}G_1 & 0 \\
	 \frac{h}{\varepsilon}G_2 & -\frac{1}{\varepsilon} & 0 &  -\frac{g}{\varepsilon}G_2 \\
	1 & 0 & -1 & 0 \\
	0 & 1 & 0 & -1}
\]
where
\beqn
G_1 &=& \GG'\left( -gx^H_1 +  \sum_{j\neq i} \alpha_{1j} x^E_j + I_1 \right)\\
G_2 &=& \GG'\left( -gx^H_2 +  \sum_{j\neq i} \alpha_{2j} x^E_j + I_2 \right)\\
h &=& a_{12}=a_{21} 
\eeqn

Therefore
\[
\varepsilon J = \Matrix{-1 & hG_1 & -gG_1 & 0 \\
	 hG_2 & -1 & 0 &  -gG_2 \\
	\varepsilon & 0 & -\varepsilon & 0 \\
	0 & \varepsilon & 0 & -\varepsilon}
\]

Since this is a $4 \times 4$ matrix, a {\em necessary} condition
for all eigenvalues to have negative real part is that the trace is negative
and the determinant is positive.
Calculations show that
\beqn
\tr(\varepsilon J) &=& -2 -2\varepsilon < 0 \\
\det(\varepsilon J) &=& \varepsilon^2[(1+gG_1)(1+gG_2)-h^2G_1G_2]
\eeqn
so the trace is always negative. However, the determinant is positive
if and only if
\[
h^2 < \left(g+\frac{1}{G_1}\right)\left(g+\frac{1}{G_2}\right)
\]
so 
\begin{equation}
\label{e:trans_stab}
|h| < \sqrt{\left(g+\frac{1}{G_1}\right)\left(g+\frac{1}{G_2}\right)}
\end{equation}
In other words: strong lateral coupling leads to transverse instability.

As parameters vary continuously, an eigenvalue
goes unstable if and only if $\det J = 0$.  Since $\varepsilon \neq 0$,
this occurs if and only if 
\[
\left(1+gG_1\right)\left(1+gG_2\right)-h^2G_1G_2 = 0
\]
That is, 
\begin{equation}
\label{e:det_zero}
h = \sqrt{\left(g+\frac{1}{G_1}\right)\left(g+\frac{1}{G_2}\right)}
\end{equation}
Therefore the synchrony subspace is stable transversely whenever~\eqref{e:trans_stab} holds.
The determinant argument above shows that it is unstable whenever
\[
h > \sqrt{\left(g+\frac{1}{G_1}\right)\left(g+\frac{1}{G_2}\right)}
\]
Therefore~\eqref{e:det_zero} determines the boundary between transverse stability
and instability of the synchrony subspace.

\subsubsection{Transverse Jacobians}

For comparison, we
also computed the eigenvalues of $J$ numerically for each primary
gait, using the parameter values of Figures 19--22 for the hop, run, jump, and walk gaits
respectively. All transverse eigenvalues (that is, the eigenvalues of $J$)
are negative on the periodic orbit, see Figure~\ref{F:biped_ev}. 

\begin{figure}[h!]
\centerline{%
\includegraphics[width=0.4\textwidth]{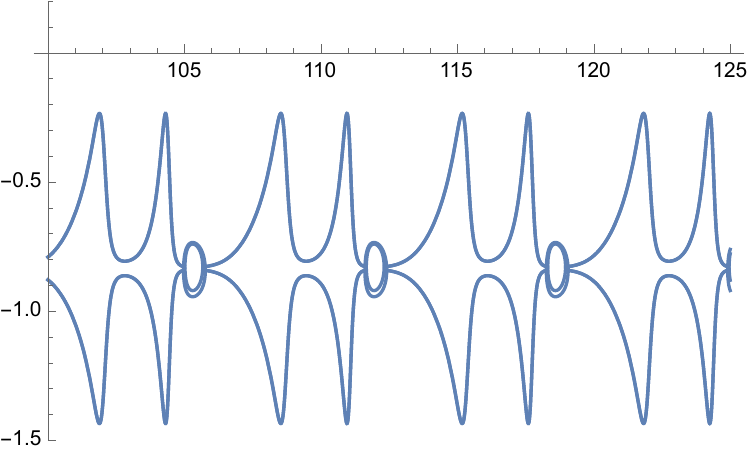} \qquad\
\includegraphics[width=0.4\textwidth]{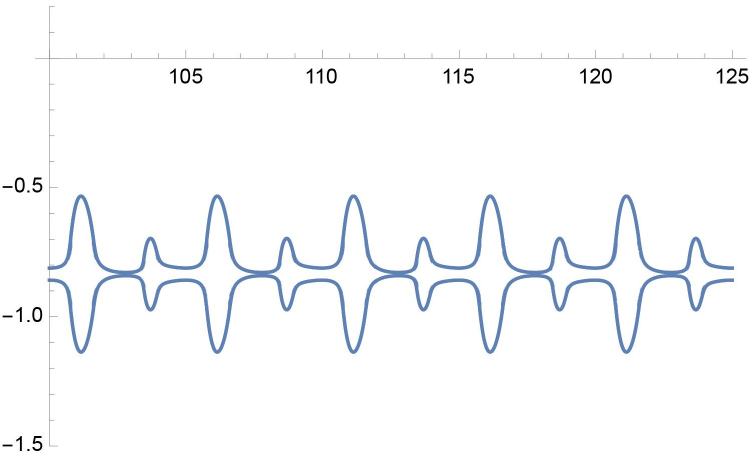}
}
\vspace{.1in}
\centerline{%
\includegraphics[width=0.4\textwidth]{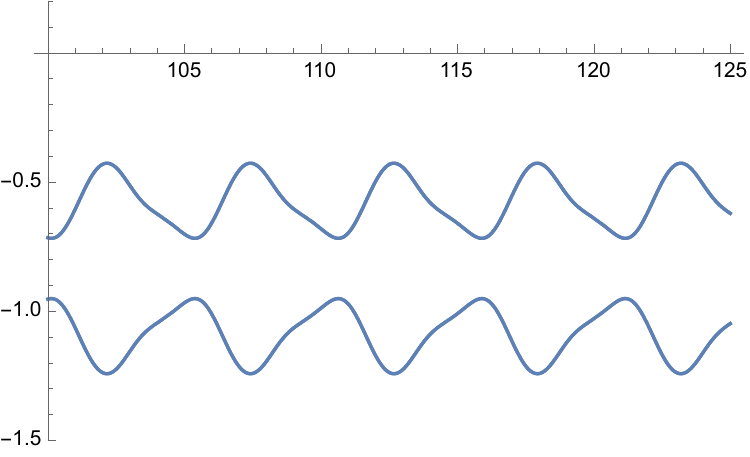} \qquad\
\includegraphics[width=0.4\textwidth]{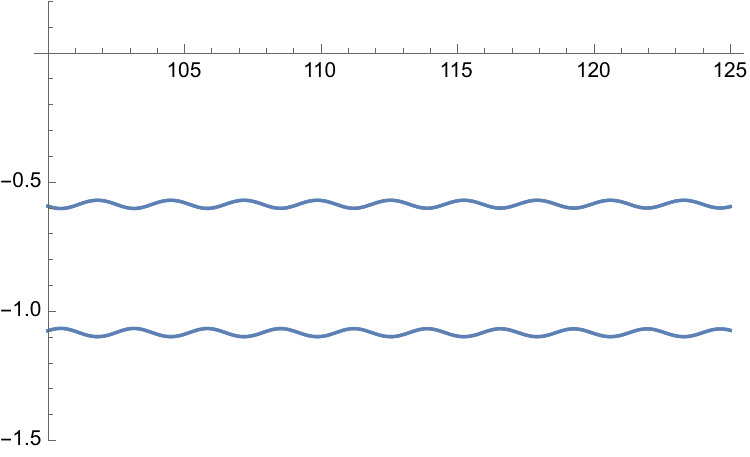}
}
\caption{Real parts of eigenvalues for the four primary biped gaits (plotted against time) 
are negative for the parameter values used in simulations in \cite[Section 14]{S14}. {\em Top left}: hop.
{\em Top right}: run. {\em Bottom left}: jump. {\em Bottom right}: walk. }
\label{F:biped_ev}
\end{figure}

In these figures the
real parts of complex conjugate pairs coincide, giving only
two values. The small ovals in `hop' correspond to real eigenvalues. 
All other eigenvalues are complex. Figure \ref{F:biped_im}, which plots the imaginary parts
of the eigenvalues, confirms these statements:
only the `hop' figure exhibits (short) intervals of zero imaginary part,
whose position matches the ovals in the real part.

\begin{figure}[h!]
\centerline{%
\includegraphics[width=0.4\textwidth]{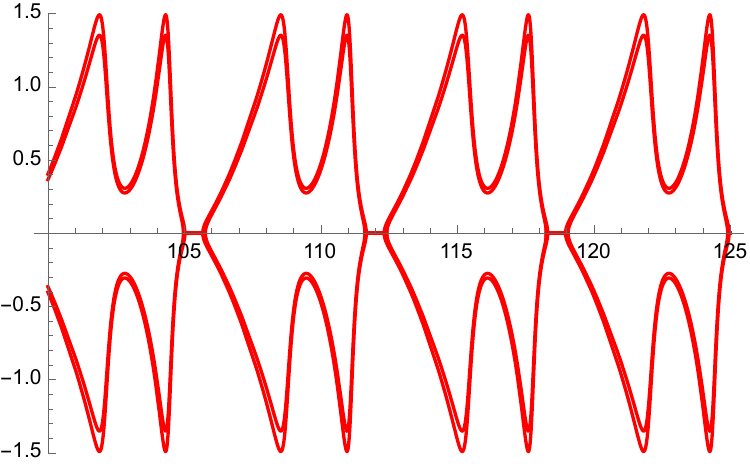} \qquad\
\includegraphics[width=0.4\textwidth]{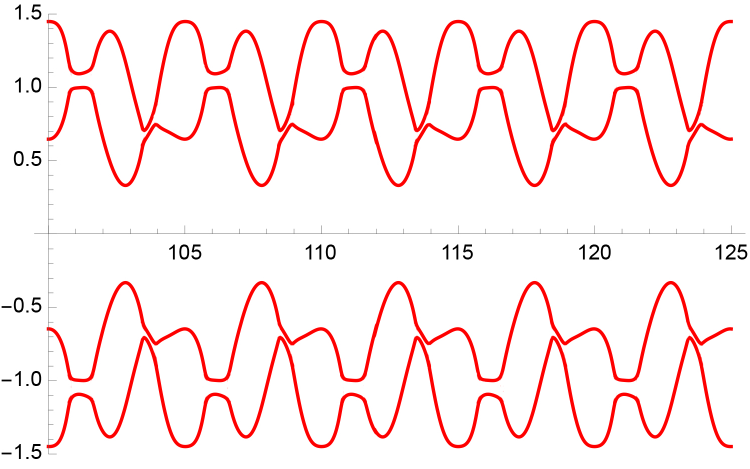}
}
\vspace{.1in}
\centerline{%
\includegraphics[width=0.4\textwidth]{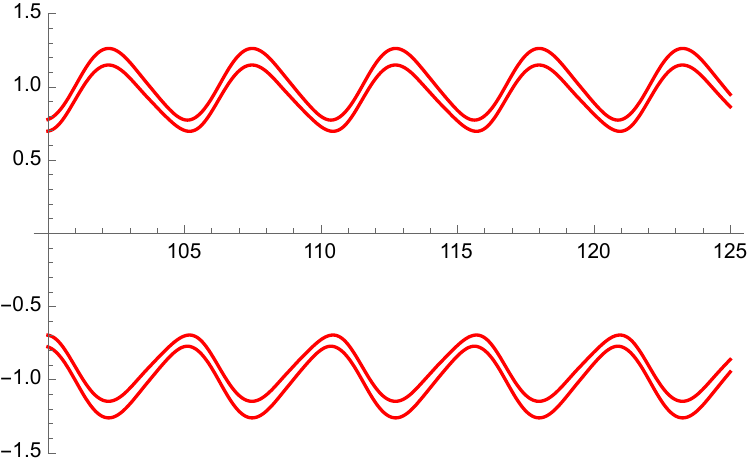} \qquad\
\includegraphics[width=0.4\textwidth]{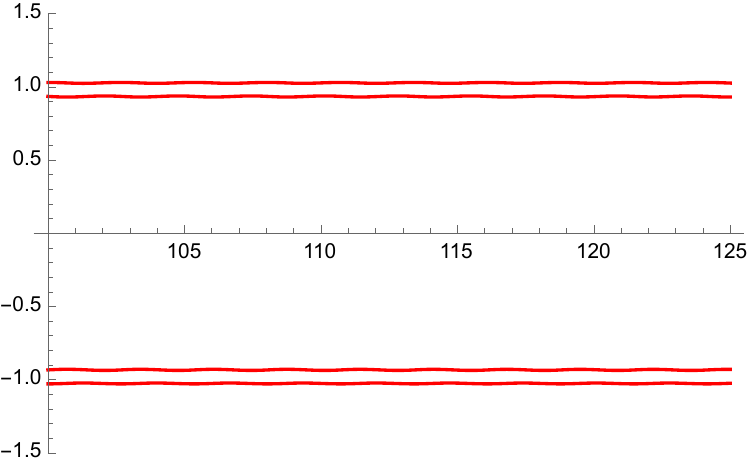}
}
\caption{Imaginary parts of eigenvalues for the four primary biped gaits (plotted against time) 
are negative for the parameter values used in simulations in \cite[Section 14]{S14}. {\em Top left}: hop.
{\em Top right}: run. {\em Bottom left}: jump. {\em Bottom right}: walk.}
\label{F:biped_im}
\end{figure}

\subsection{CPG Floquet Multipliers}

We computed the Floquet multipliers for the CPG, for the four gaits
{\em hop, run, jump, walk}, using the parameter values of \cite{S14}.
For all gaits we take $a=1,b=8,c=1,\eps=0.67$, $I=1.1$, $g=1.8$. (Actually  \cite{S14}
uses $I=0.8$ for the hop, but we replace this here by $1.1$.
The other parameters are:
\beqn
\mbox{hop}&:&\quad \alpha = 0.5 \quad \beta = 0.6 \quad \gamma= 0.8\\
\mbox{run}&:&\quad \alpha = -0.5 \quad \beta = -0.6 \quad \gamma= 0.8\\
\mbox{jump}&:&\quad \alpha = -0.5 \quad \beta = 0.6 \quad \gamma= -0.8\\
\mbox{walk}&:&\quad \alpha = 0.5 \quad \beta = -0.6 \quad \gamma= -0.8
\eeqn

The results for the 1-node module
of Figure \ref{F:4CPGcascade} (top)
are listed below in Table \ref{T:FloqBIPED}, which
also lists the period and the transverse Floquet multipliers.

The period and 
CPG Floquet multipliers for the 2-node module of 
Figure \ref{F:4CPGcascade} (bottom) 
are the same as in Table \ref{T:FloqBIPED}.
The transverse Floquet multipliers for the 2-node module 
are listed in
Table \ref{T:FloqBIPED2module}.

\subsection{Transverse Floquet Multipliers, $1$-Node Module}

We begin with Figure~\ref{F:4CPGcascade} (top),
which is a genuine feedforward lift. It is enough to consider
the transverse Floquet equation for node $5$. Use the notation $\alpha,\beta,\gamma$
of~\cite{S14} for connection strengths, and let the CPG periodic orbit be
\beqn
&&(x_1^E,x_2^E,x_3^E,x_4^E,x_1^H,x_2^H,x_3^H,x_4^H) =\\
&&\quad (a_1^E(t),a_2^E(t),a_3^E(t),a_4^E(t),a_1^H(t),a_2^H(t),a_3^H(t),a_4^H(t)) 
\eeqn
this is:
\begin{equation}
\label{E:node5}
\begin{array}{rcl}
\eps \dot x^E_5 &=& -x_5^E + \GG(-gx_5^H + \beta a_2^E(t)+\gamma a_3^E(t)+\alpha a_4^E(t)+I)\\
\dot x^H_5 &=& x^E_5-x^H_5
\end{array}
\end{equation}
Here the inputs $x_3^E(t), x_4^E(t), x_5^E(t)$ are evaluated on the periodic orbit of the CPG using \eqref{E:4noderateeq}, and the parameters $\eps,\alpha,\beta,\gamma, g, I$
refer to that equation.

\begin{table}[h!]
\small
\begin{center}
\begin{tabular}{|l|l||l|l||l|l|}
\hline
gait & $T$ & CPG multipliers  & abs. & trans. multipliers & abs. \\
\hline
\hline
Hop & $6.646$ & $0.999^*$  & $0.999$ & $0.00172$ & $0.00172$\\		
stable& & $0.00117$ & $0.00117$ &0.0000188& $0.0000188$ \\
 & & $0.000946 $ & 0.000946 & &  \\
& & $0.000400 $ & 0.000400 & &  \\
& & $0.000258$  & $0.000258$  & & \\
& & $0.0000143$  & $0.0000143$  & & \\
& & $4.27\times10^{-6}$  & $4.27\times10^{-6}$  & & \\
& & $2.35 \times10^{-6}$  & $2.35 \times10^{-6}$  & & \\

\hline
Run& $4.991$& $1$ & $1$ & $0.00685$ & $0.00685$\\
stable& & $0.150 $ & $0.150$ & 0.000571&0.000571 \\
& & $0.00714$ & $0.00714$ & & \\
& & $0.00110 $ & $0.00110$ & & \\
& & $0.000322 $ & $0.000322$ & & \\
& & $0.000166 $ & $0.000166$ & & \\
& & $0.0000773 $ & $0.0000773$ & & \\
& & $0.0000497 $ & $0.0000497$ & & \\
\hline
Jump & $5.257$ & $1$ & $1$ & $0.00522$ & $0.00522$\\
stable& & $0.505$ & $0.505$ &0.000389 &0.000389 \\
& & $ 0.000180 \pm 0.000387 \,\ii$ & $ 0.000427$ & & \\
& & $0.000180 \pm 0.000111 \,\ii$ & $0.000212$ & & \\
& & $0.0000580 \pm 0.0000273 \,\ii$ & $0.0000641$ & & \\
\hline
Walk &  $5.368$ & $0.998^*$ & $0.998$ & $0.00470$& $0.00470$ \\
stable&  & $ 0.916  $ & $ 0.916$ & 0.000325& 0.000325\\
& & $-8.09\times10^{-6} \pm  0.000424 \,\ii$ & $0.000424$ & & \\
& & $0.0000531 \pm 0.0000914 \,\ii$ & $0.000105$ & & \\
& & $0.0000452 \pm 0.0000322 \,\ii$ & $0.0000555$ & & \\
\hline
\end{tabular}
\caption{CPG and transverse Floquet multipliers for the four gaits of the biped
model with a 1-node module, Figure~\ref{F:4CPGcascade} (top).
The CPG periodic orbit and the lifted periodic orbit are stable for all four parameter sets.
All entries stated to three significant figures after the decimal point. $^*$These entries are presumably 
$1$ with small numerical error. See Remark~\ref{R:starred}.}
\label{T:FloqBIPED}
\end{center}
\end{table}

Numerical computation shows that all four gaits are Floquet-stable at the 
parameter values stated above.
The computed data (truncated to 3 significant figures, four for the period) 
are shown in Table~\ref{T:FloqBIPED}.

The transverse Floquet multipliers are generally smaller by a few orders of 
magnitude than the CPG Floquet multipliers, indicating a high degree
of stability of the lifted periodic orbit.

\begin{remark}\em
\label{R:starred}
As regards the starred entries in Table~\ref{T:FloqBIPED}, it 
is observed in \cite{LE02} that in numerical
integration of Floquet equations: `The accuracy of the computed trivial multiplier is not always comparable to the accuracy of the computed periodic solution and the accuracy of the other computed multipliers.'
\end{remark}

\subsection{Transverse Floquet Multipliers, 2-Node Module}

For the 2-node module, with the same parameters for the CPG,
the Floquet multipliers for the CPG are unchanged, so the periodic orbits
are Floquet-stable. The transverse Floquet multipliers change, and we
compute them to get table \ref{T:FloqBIPED2module}.
Again, all gaits are Floquet-stable for these parameters, and
the transverse Floquet multipliers are generally smaller by a few orders of 
magnitude than the CPG Floquet multipliers, indicating a high degree
of stability of the lifted periodic orbit.

The main conclusion is that in this case, replacing feedforward arrows by lateral arrows
does not affect the stability significantly.

\begin{table}[h!]
\small
\begin{center}
\begin{tabular}{|l||l|l|}
\hline
gait & transverse multipliers & absolute value \\

\hline
\hline
Hop & $0.0184$ & $0.0184$ \\		
stable & $0.000322$ & $0.000322$ \\	
 & $0.000216$ & $0.000216$ \\
  & $3.16\times10^{-6}$ & $3.16\times10^{-6}$ \\		
\hline
Run & $0.0281$ & $0.0281$ \\
stable & $0.0190$ & $0.0190$ \\
 & $0.000283$ & $0.000283$ \\
 & $0.000102$ & $0.000102$ \\
\hline
Jump  & $0.000576 \pm 0.0182 \,\ii$ & $0.0182$ \\
stable & $0.0000931 \pm 0.0000614 \,\ii$ & $0.000111$ \\  
\hline
Walk  & $-0.00498 \pm 0.0199 \,\ii$ & $0.0205$ \\
stable & $0.0000502 \pm 0.0000559 \,\ii$ & $0.0000751$ \\
\hline
\end{tabular}
\caption{Transverse Floquet multipliers for the four gaits of the biped
model with a 2-node module, Figure~\ref{F:4CPGcascade} (bottom).
The CPG periodic orbit and the lifted periodic orbit are stable for all four parameter sets.
All entries stated to three significant figures after the decimal point. }
\label{T:FloqBIPED2module}
\end{center}
\end{table}

\subsection{Parameters for which Proposition \ref{P:Liap2} Applies}
\label{S:PCA}

To facilitate comparison with  \cite{S14} we have used the parameters
values of that paper, where $g=1.8$ and the gain function is \eqref{E:standardgain}.
However, this 
this value is slightly too big to apply Proposition \ref{P:trans_stab_Floq},
which requires $g < 1.5$.
We now exhibit different simulations if the four primary gaits for $g=1.4$, so that 
Proposition \ref{P:Liap2} applies.

We take the following parameter values for the four gaits:
\begin{eqnarray}
 \label{E:RATEparamsNEW1}
\mbox{\em Hop}: && I=0.7 \quad \alpha = 0.5 \quad  \beta = 0.6 \quad \gamma = 0.8\quad g=1.4 \quad \eps = 0.5\\
\label{E:RATEparamsNEW2}
\mbox{\em Jump}: && I=1.1 \quad \alpha = -0.5 \quad  \beta = 0.6 \quad \gamma = -0.8\quad g=1.4 \quad \eps = 0.5\\
\label{E:RATEparamsNEW3}
\mbox{\em Run}: && I=1.1 \quad \alpha = -0.5 \quad  \beta = -0.6 \quad \gamma = 0.8\quad g=1.4 \quad \eps = 0.5\\
\label{E:RATEparamsNEW4}
\mbox{\em Walk}: && I=1.1 \quad \alpha = 0.5 \quad  \beta = -0.6 \quad \gamma = -0.8\quad g=1.4 \quad \eps = 0.5
\end{eqnarray}

\begin{figure}[h!]
\centerline{
\includegraphics[width=.18\textwidth]{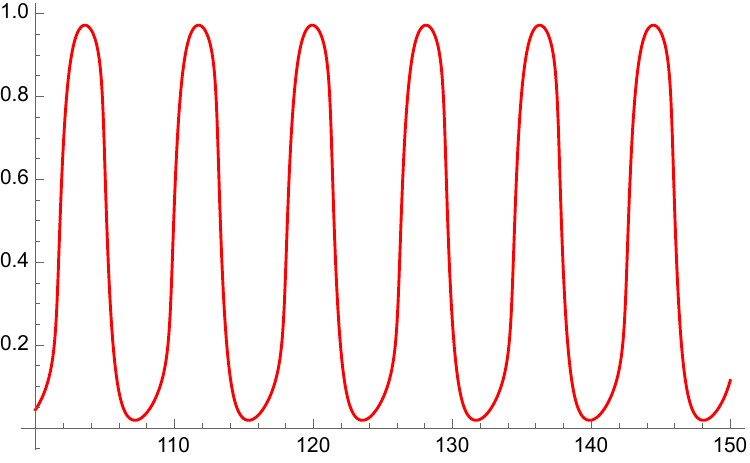} \ 
\includegraphics[width=.18\textwidth]{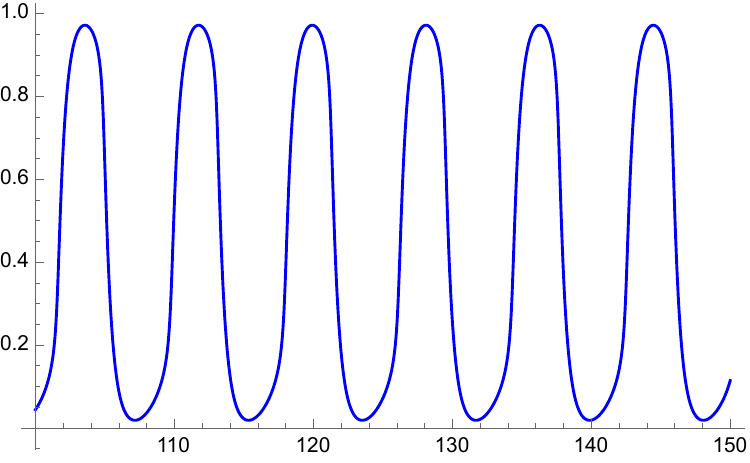} \ 
\includegraphics[width=.18\textwidth]{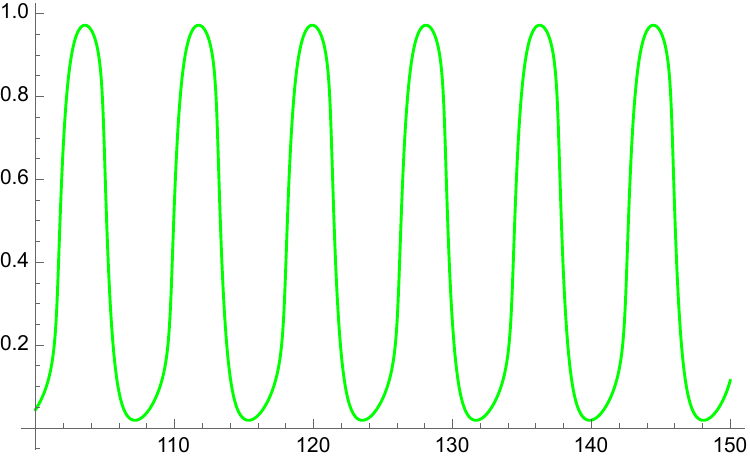} \ 
\includegraphics[width=.18\textwidth]{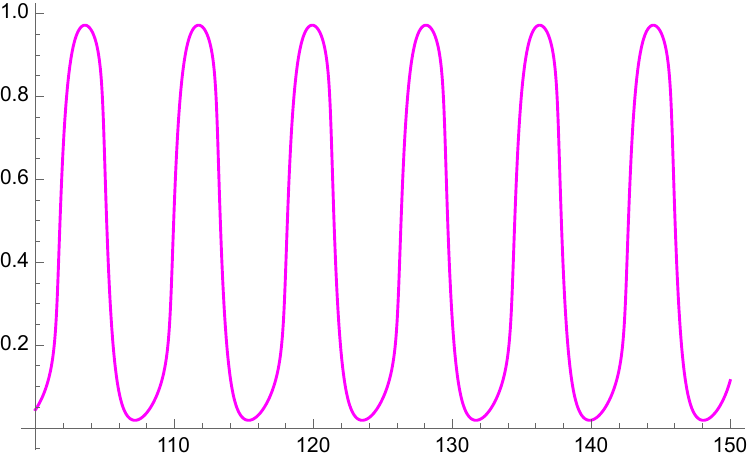} \ 
\includegraphics[width=.18\textwidth]{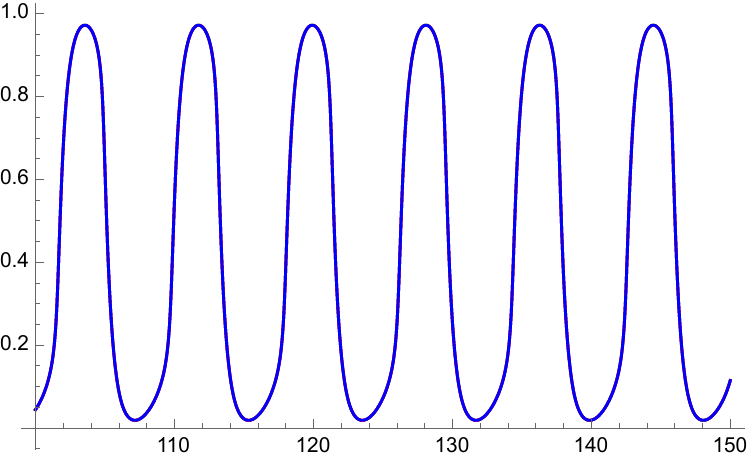}
}
\centerline{
\includegraphics[width=.18\textwidth]{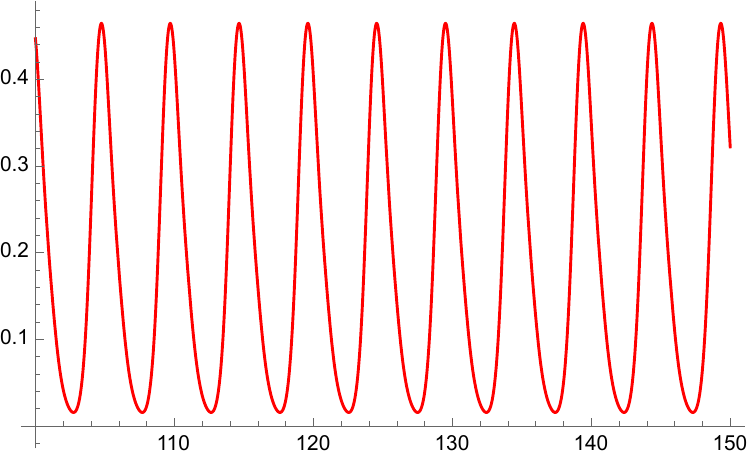} \ 
\includegraphics[width=.18\textwidth]{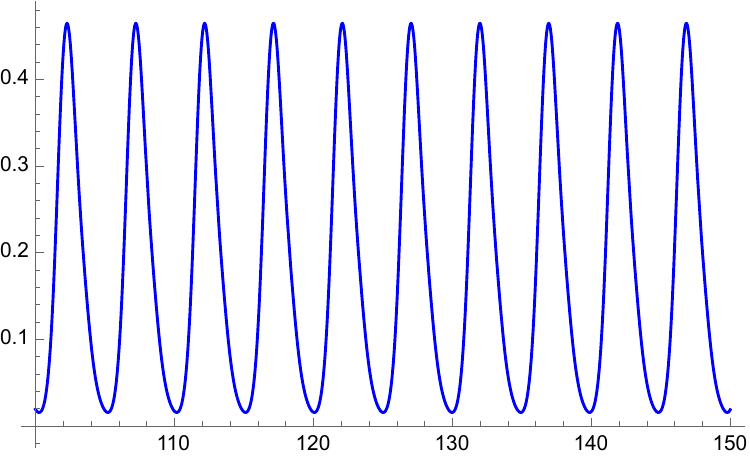} \ 
\includegraphics[width=.18\textwidth]{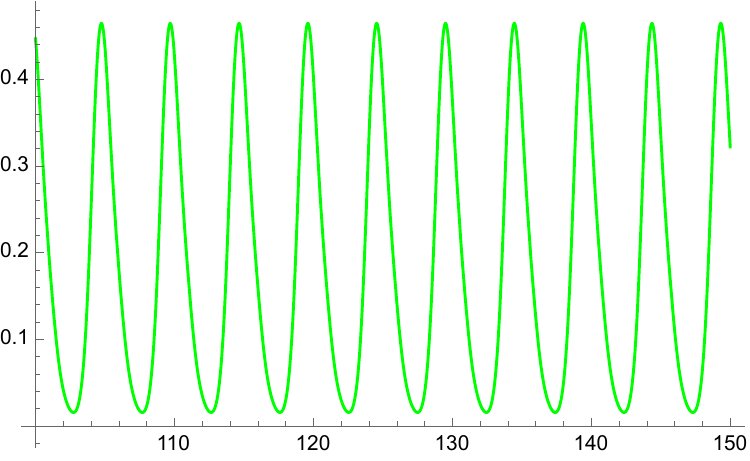} \ 
\includegraphics[width=.18\textwidth]{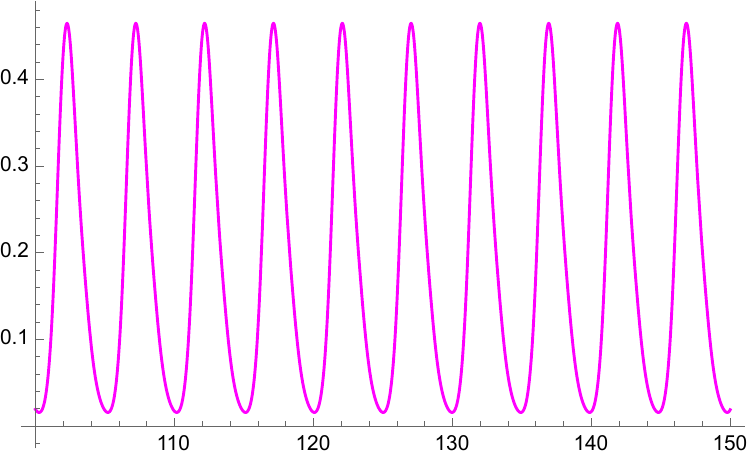} \ 
\includegraphics[width=.18\textwidth]{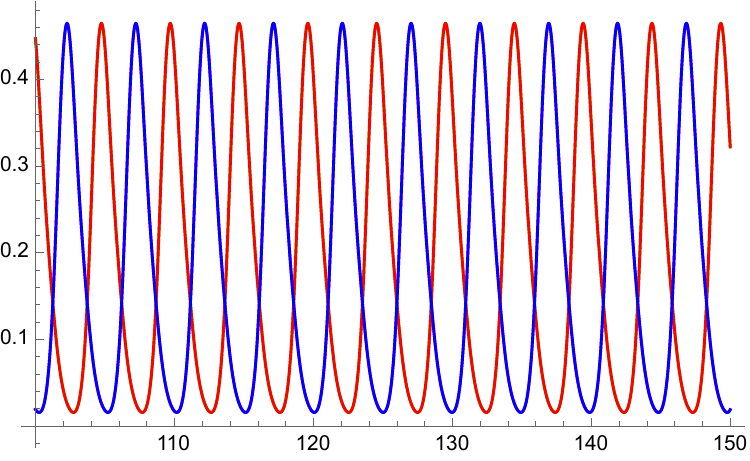}
}
\centerline{
\includegraphics[width=.18\textwidth]{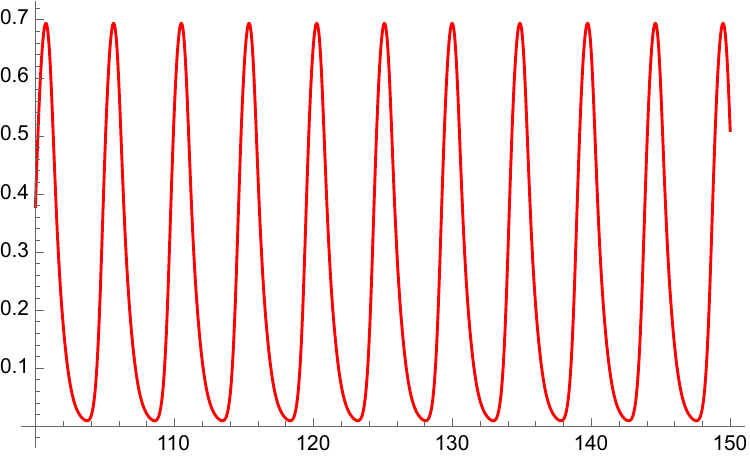} \ 
\includegraphics[width=.18\textwidth]{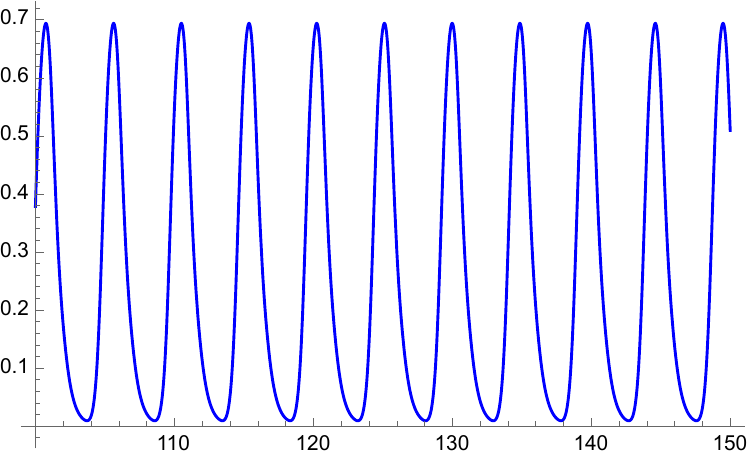} \ 
\includegraphics[width=.18\textwidth]{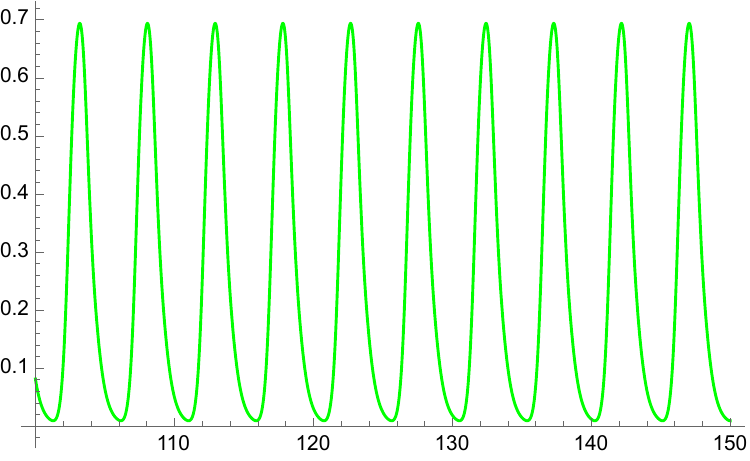} \ 
\includegraphics[width=.18\textwidth]{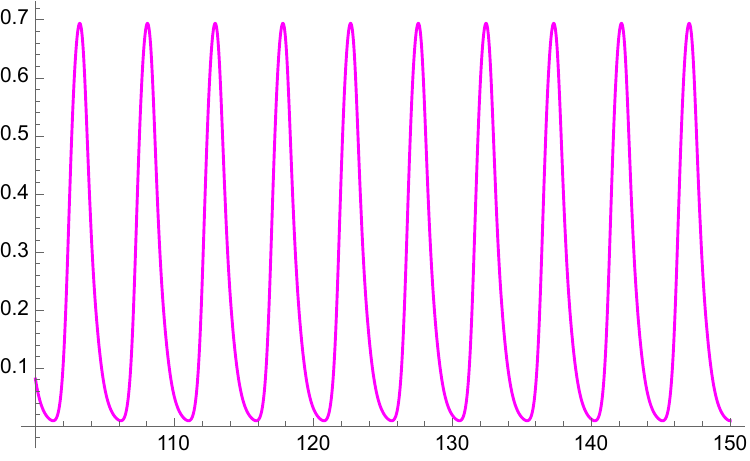} \ 
\includegraphics[width=.18\textwidth]{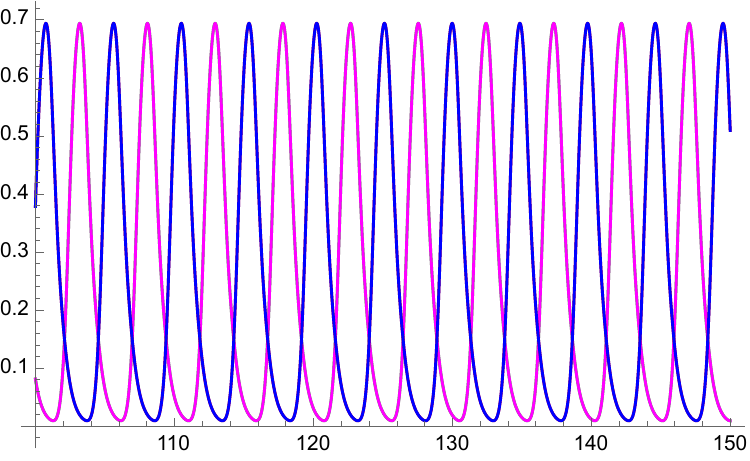}
}
\centerline{
\includegraphics[width=.18\textwidth]{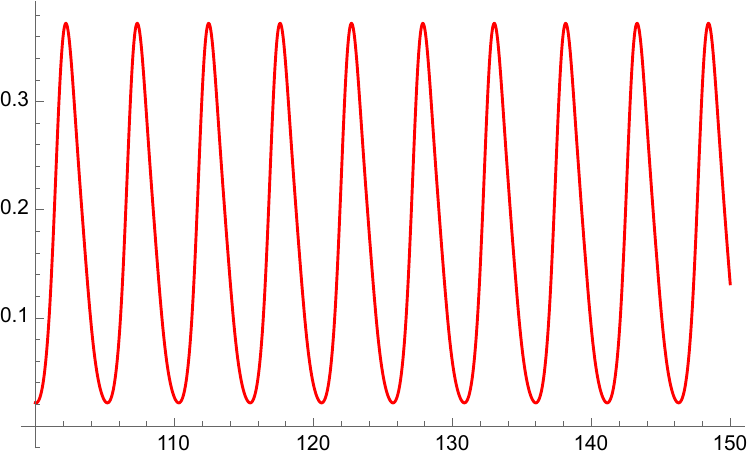} \ 
\includegraphics[width=.18\textwidth]{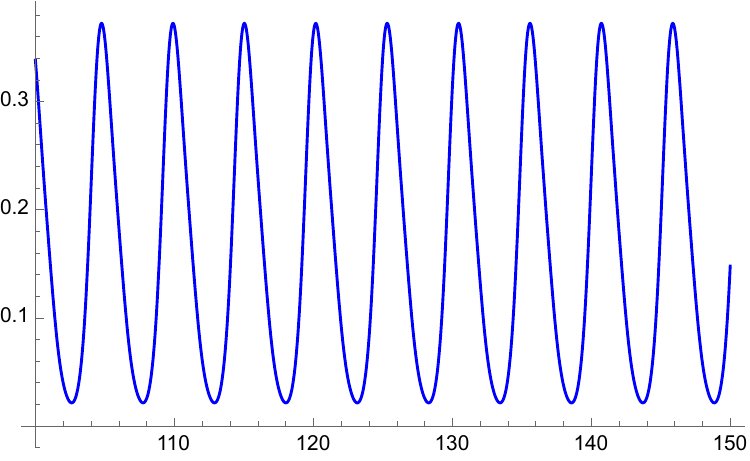} \ 
\includegraphics[width=.18\textwidth]{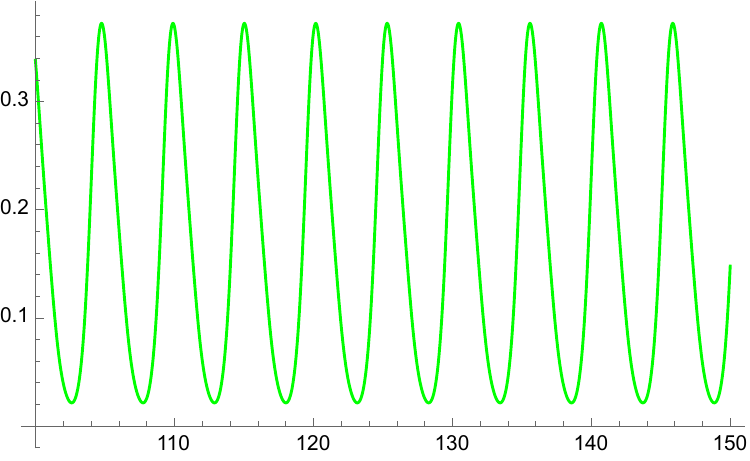} \ 
\includegraphics[width=.18\textwidth]{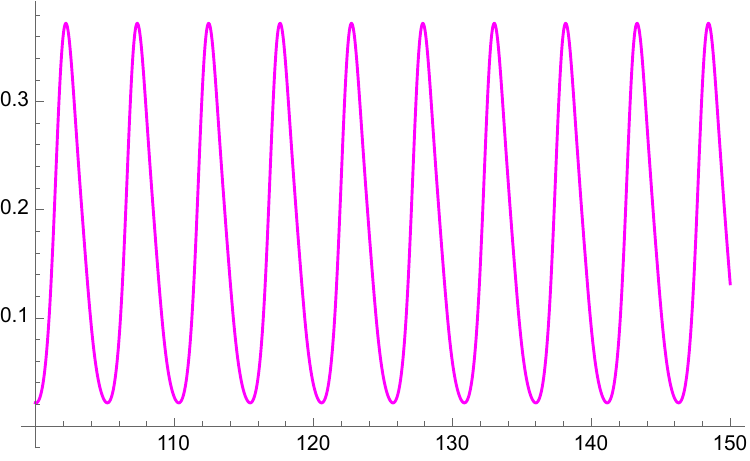} \ 
\includegraphics[width=.18\textwidth]{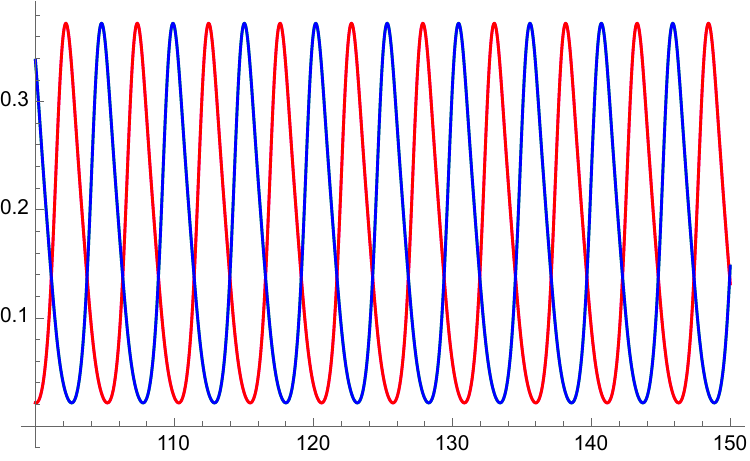}
}
\caption{Time series of selected activity variables for a rate model of 
Figure \ref{F:4CPGcascade} (top). Each row shows $x^E_i$ for $1 \leq i \leq 4$,
followed by all nodes in the chain (up to node 8) superposed.
{\em From top to bottom}: hop, jump, run, walk.
}
\label{F:GAITSg=1.4}
\end{figure}

Proposition \ref{P:trans_stab_Floq} does not apply with these parameters
because $\eps = 0.5$ and this requires $g < 2.22474/2 = 1.11237$.
However, Proposition \ref{P:Liap2} does apply, since when $\eps =0.5$ 
this requires $g < 4.82843/2 = 2.41411$, so these gaits are transversely 
Liapunov Stable for the stated parameter values.

%%\ignore{%%%%%%%%%%%%%%%% KEEP THIS FOR THE MOMENT
More refined estimates are also possible, because the range of values
taken by $\eta(t)$ in \eqref{E:eta(t)} need not
include $x=1$, where $\GG'(x)$ attains its maximum value of $2$s. 
In Figure \ref{F:GAITSg=1.4} all oscillations are less than 1. The
maxima are roughly
\beqn
\mbox{\em hop} &:& 0.97 \\
\mbox{\em jump} &:& 0.46 \\
\mbox{\em run} &:& 0.69 \\
\mbox{\em walk} &:& 0.37      
\eeqn 
Since $\GG'(x)$ is monotone increasing for $x < 1$ the 
corresponding bounds for $g$ are $1.52$, $28.95$, $5.25$, and $58.67$.
Thus the estimates needed for the jump, run, and walk are more generous.
With other values of $\eps$, modified estimates of this kind can
prove transverse Floquet stability when Proposition \ref{P:trans_stab_Floq}
does not apply. We do not attempt to state such modified estimates in general,
but there is clearly scope for further improvements.
%}%%%%%%%%%%%%%%%%%%%%%%%%%

\ignore{%%%%%%%%%%%%%%%%%%%%%
\subsection{Challenges to the Prevailing Paradigm} 
\label{S:CPP}
\RED{ FOOTNOTE}\footnote{Speculative. Think it over. Move? Delete?}

Recent data on neuron activity in the spinal cord of turtles \cite{LPVB22} suggests that the
prevalent paradigm for control of animal locomotion may require modification.
The authors remark that `As flexor and extensor muscle activities alternate 
during rhythmic movements such as walking, it is often assumed that the 
responsible neural circuitry is similarly exhibiting alternating activity.'  Their
measurements show different behaviour: 
 a population of neurons in the spinal cord performs a low-dimensional
`rotation' in neural space, as shown by Principal Component Analysis. Thus the neural
activity in the spinal cord cycles continuously through all phases.

The model of propagation presented here assumes that a CPG
generates the rhythmic
patterns, but it is also consistent with
the above remarks. Measurement of any neuron in the chain would
show a periodic cycle, and measurements of other neurons would show
the simultaneous presence of many different phases of the same cycle
at different locations.  
The time series show that the networks proposed here can 
generate continuous periodic signals for many different
neuron equations, and do not require the activity of any set of neurons to 
alternate between `on' and `off' states. Instead, all neurons in the chain participate in the
dynamics at all times.
}%%%%%%%%%%%%%%%%%%%%%END BLUE

\section{Conclusions}
\label{S:C}

Regular phase patterns are common in networks of dynamical systems 
with suitable symmetries, which in particular suggests modeling animal gait patterns 
using symmetric networks. 
Animal gaits are widely thought to be created by a central pattern
generator (CPG), located in the spinal column. Several model
CPGs have been studied previously from the viewpoint
of symmetric dynamics.

Networks with feedforward lift structure (CPG + chain) naturally 
propagate synchronous or phase-synchronous states along linear chains.
This type of network topology suggests new and physiologically plausible
networks for animal locomotion, which we analysed in the case of bipeds.
We studied two such feedforward lifts using a CPG proposed by
Pinto and Golubitsky \cite{PG06} and extending it along a chain
of arbitrary length formed from `modules' that are copies of the CPG.

An important issue is the stability of the propagating signals. This 
splits into stability under synchrony-preserving perturbations (stability for the CPG)
and  stability under synchrony-breaking perturbations
(transverse stability). This decomposition was discussed
in \cite{SW23a} for several stability notions, and was specialized here to the
two feedforward lifts. Using a rate model, we established Floquet stability 
(for suitable parameters) using numerical calculations of CPG and 
transverse Floquet multipliers.
Liapunov stability was established analytically using two Liapunov functions.

Feedforward lifts provide a simple, effective, and robust way to propagate signals
with specific synchrony and phase patterns in a stable manner. They
suggest biologically plausible topologies for neuronal networks that
propagate such signals.

%\paragraph{Acknowledgements}

%We thank Peter Ashwin, Marty Golubitsky, and John Guckenheimer for helpful discussions.

\end{document}